\documentclass[10pt,reqno]{amsart}
\usepackage{amsmath,amsfonts,amssymb,amsthm}
\usepackage[alphabetic]{amsrefs}
\usepackage{url}
\usepackage{graphicx,color}
\usepackage{scrextend}
\usepackage{cases}
\usepackage{enumitem}
\usepackage{systeme}
\usepackage{hyperref}
\hypersetup{
    colorlinks=true,
    linkcolor=blue,
    filecolor=magenta,
    urlcolor=blue,
    citecolor=blue,
}
\usepackage{subcaption}

\usepackage{xcolor,colortbl}
\definecolor{grayDisabled}{rgb}{0.75, 0.75, 0.75}
\newcolumntype{d}{>{\columncolor{grayDisabled}}r}
\newcolumntype{e}{>{\columncolor{grayDisabled}}l}

\makeatletter
\newcommand{\Vast}{\bBigg@{4}}
\makeatother

\newtheorem{theorem}{Theorem}[section]
\newtheorem{lemma}[theorem]{Lemma}

\newtheorem*{theorem*}{Theorem}

\theoremstyle{definition}
\newtheorem{definition}{Definition}[section]

\theoremstyle{remark}
\newtheorem*{remark}{Remark}

\usepackage{caption}

\begin{document}

\title[Recurrence relations in (s,t)-uniform simplicial complexes]{Recurrence relations in (s,t)-uniform simplicial complexes}

\author[Ioana-Claudia Laz\u{a}r]{
Ioana-Claudia Laz\u{a}r\\
Politehnica University of Timi\c{s}oara, Dept. of Mathematics,\\
Victoriei Square $2$, $300006$-Timi\c{s}oara, Romania\\
E-mail address: ioana.lazar@upt.ro}

\date{}

\hyphenation{i-so-pe-ri-me-tric}

\begin{abstract}

We introduce $(s,t)$-uniform simplicial complexes.
We show that the lengths of spheres in minimal filling diagrams associated to loops in such complexes are the terms of certain recurrence relations.
We study the limit of the ratio of the area of such spheres over their length as the radii of spheres grow.
Besides we compute the average Gaussian curvature for vertices inside these spheres.

\hspace{0 mm} \textbf{2010 Mathematics Subject Classification}: 05E45, 20F67, 11B37.

\hspace{0 mm} \textbf{Keywords}: simplicial complex, minimal filling diagram, $t$-uniformness, recurrence relation, Gaussian curvature
\end{abstract}

\pagestyle{myheadings}

\markboth{}{}

\vspace{-10pt}

\maketitle

\section{Introduction}

Isoperimetric inequalities relate the length of closed curves to the infimal area of the discs which they bound.
Every closed loop of length $L$ in the Euclidean plane bounds a disc whose area is less than $L^{2} / 4 \pi$, and this bound is optimal.
Thus one has a quadratic isoperimetric inequality for loops in Euclidean space.
In contrast, loops in real hyperbolic space satisfy a linear isoperimetric inequality: there is a constant $C$ such that every closed loop of length $L$ in hyperbolic space bounds a disc whose area is less than or equal to $C \cdot L$.

With a suitable notion of area, a geodesic space $X$ is $\delta$-hyperbolic if and only if loops in $X$ satisfy a linear isoperimetric inequality (see \cite{BH}, chapter $III.H$, page $417$ and page $419$).
For loops in arbitrary CAT(0) spaces, however, there is a quadratic isoperimetric inequality (see \cite{BH}, chapter $III.H$, page $414$).
Osajda introduced in \cite{O-8loc} a local combinatorial condition called $8$-location implying Gromov hyperbolicity of the universal cover (see \cite{L-8loc}).
A related curvature condition, called $5/9$-condition, also implies Gromov hyperbolicity (see \cite{L-8loc2}).
Both $8$-located complexes and $5/9$-complexes satisfy therefore, under the additional hypothesis of simply connectedness, a linear isoperimetric inequality.

One can also express curvature using a condition called local $k$-largeness which was introduced independently by Chepoi \cite{Ch} (under the name of bridged complexes) and by Januszkiewicz-Swiatkowski \cite{JS1}.
A flag simplicial complex is locally $k$-large if its links do not contain essential loops of length less than $k, k \ge 4$.
Cycles in systolic complexes satisfy a quadratic isoperimetric inequality (see \cite{JS1}).
In \cite{E1} explicit constants are provided presenting the optimal estimate on the area of a systolic disc.
In systolic complexes the isoperimetric function for $2$-spherical cycles (the so called second isoperimetric function) is linear (see \cite{JS2}).
In \cite{ChaCHO} it is shown that meshed graphs (thus, in particular, weakly modular graphs) satisfy a quadratic isoperimetric inequality.

The purpose of the current paper is to extend some results obtained in \cite{L-t-unif} on $t$-uniform simplicial complexes.
We define $(s,t)$-uniform simplicial complexes, $s,t \geq 4$.
We take into account only the lengths of spheres, not of all loops.
We show that for $s,t \geq 6, t$ even except for $s=t=6$, there is a connection between the lengths of spheres $S_{n}^{s,t}$ in minimal filling diagrams associated to loops in $(s,t)$-uniform complexes and the terms of certain constant-recursive sequences.
For $s = t = 6$, the complex is $6$-uniform and therefore flat.
So it fulfills a quadratic isoperimetric inequality.
Besides we find other simplices on spheres that satisfy the same recurrence relation (see Theorem $3.1$).
For the hyperbolic case, we find the limit of the ratio $\cfrac{A_{n}^{s,t}}{|S_{n}^{s,t}|}$ as $n \rightarrow \infty$.
In contrast to $t$-uniform complexes (see \cite{L-t-unif}), for $(s,t)$-uniform complexes this limit can no longer be used as the best constant for the isoperimetric inequality.
We also study the average Gaussian curvature for vertices inside spheres as the radii of spheres grow.

\textbf{Acknowledgements}.
The author would like to thank Damian Osajda for introducing her to the subject.
This work was partially supported by the grant $346300$ for IMPAN from the Simons Foundation and the matching $2015-2019$ Polish MNiSW fund.

\section{Preliminaries}

\subsection{Simplicial complexes}

Let $X$ be a simplicial complex.
We denote by $X^{(k)}$ the $k$-skeleton of $X, 0 \le k < \dim X$.
A subcomplex $L$ in $X$ is called \emph{full} as a subcomplex of $X$ if any simplex of $X$ spanned by a set of vertices in $L$, is a simplex of $L$.
For a set $A = \{ v_{1}, ..., v_{k} \}$ of vertices of $X$, by $\langle A \rangle$ or by $\langle v_{1}, ..., v_{k} \rangle$ we denote the \emph{span} of $A$, i.e. the smallest full subcomplex of $X$ that contains $A$.
We write $v \sim v'$ if $\langle v,v' \rangle \in X$ (it can happen that $v = v'$).
We write $v \nsim v'$ if $\langle v,v' \rangle \notin X$.
We call $X$ {\it flag} if any finite set of vertices which are pairwise connected by edges of $X$, spans a simplex of $X$.

A {\it cycle} ({\it loop}) $\gamma$ in $X$ is a subcomplex of $X$ isomorphic to a triangulation of $S^{1}$.
A \emph{full cycle} in $X$ is a cycle that is full as a subcomplex of $X$.
A $k$-\emph{wheel} in $X$ $(v_{0}; v_{1}, ..., v_{k})$ (where $v_{i}, i \in \{0,..., k\}$ are vertices of $X$) is a subcomplex of $X$ such that $(v_{1}, ..., v_{k})$ is a full cycle and $v_{0} \sim v_{1}, ..., v_{k}$.
The \emph{length} of $\gamma$ (denoted by $|\gamma|$) is the number of edges of $\gamma$.

We define the \emph{metric} on the $0$-skeleton of $X$ as the number of edges in the shortest $1$-skeleton path joining two given vertices.

Let $\sigma$ be a simplex of $X$.
The \emph{link} of $X$ at $\sigma$, denoted $X_{\sigma}$, is the subcomplex of $X$ consisting of all simplices of $X$ which are disjoint from $\sigma$ and which, together with $\sigma$, span a simplex of $X$.
We call a flag simplicial complex \emph{k-large} if there are no full $j$-cycles in $X$, for $j < k$.
We say $X$ is \emph{locally k-large} if all its links are $k$-large.
We call a vertex of $X$ \emph{k-large} if its link is $k$-large.

\begin{definition}\label{def:simplicial-map}
A \emph{simplicial map} $f : X \rightarrow Y$ between simplicial complexes $X$ and $Y$ is a map which sends vertices to vertices, and whenever vertices $v_{0}, ..., v_{k} \in X$ span a simplex $\sigma$ of $X$ then their images span a simplex $\tau$ of $Y$ and we have $f(\sigma) = \tau$.
Therefore a simplicial map is determined by its values on the vertex set of $X$.
A simplicial map is \emph{nondegenerate} if it is injective on each simplex.
\end{definition}

\begin{definition}\label{def:filling-diagram}
Let $\gamma$ be a cycle in $X$.
A \emph{filling diagram} for $\gamma$ is a simplicial map $f : D \rightarrow X$ where $D$ is a triangulated $2$-disc, and $f | _{\partial D}$ maps $\partial D$ isomorphically onto $\gamma$.
We denote a filling diagram for $\gamma$ by $(D,f)$ and we say it is
\begin{itemize}
    \item \emph{minimal} if $D$ has minimal area (it consists of the least possible number of $2$-simplices among filling diagrams for $\gamma$);
    \item \emph{nondegenerate} if $f$ is a nondegenerate map;
    \item \emph{locally k-large} if D is a locally k-large simplicial complex.
\end{itemize}
\end{definition}

\begin{lemma}\label{lemma:filling-diagram}
Let $X$ be a simplicial complex and let $\gamma$ be a homotopically trivial loop in $X$.
Then:
\begin{enumerate}
    \item there exists a filling diagram $(D, f)$ for $\gamma$ (see \cite{Ch} - Lemma $5.1$, \cite{JS1} - Lemma $1.6$ and \cite{Pr} - Theorem $2.7$);
    \item any minimal filling diagram for $\gamma$ is simplicial and nondegenerate (see \cite{Ch} - Lemma $5.1$, \cite{JS1} - Lemma $1.6$, Lemma $1.7$ and \cite{Pr} - Theorem $2.7$).
\end{enumerate}
\end{lemma}

Let $D$ be a simplicial disc.
We denote by $C$ the cycle bounding $D$ and by $\rm{Area} C$ the area of $D$.
We denote by $V_{i}$ and $V_{b}$ the numbers of internal and boundary vertices of $D$, respectively.
Then: $\rm{Area} C = 2 V_{i} + V_{b} - 2 = |C| + 2 (V_{i} - 1)$ (Pick's formula).
In particular, the area of a simplicial disc depends only on the number of its internal and boundary vertices.

We call \emph{$s$-polygon} a polygon which has $s$ edges.
We call \emph{$s$-vertex} a vertex which has $s$ neighbours.
Note that in a disc $s$-vertices are interior vertices.
We call \emph{$(s,t)$-edge} an edge spanned by an $s$-vertex and a $t$-vertex.

\begin{definition}\label{def:t-uniform}
Let $X$ be a flag, simply connected simplicial complex and let $\gamma$ be a loop in $X$.
We call $X$ \emph{(s,t)-uniform}, $s, t \ge 4$, if the complex fulfills one of the following conditions:
\begin{enumerate}
    \item if $s = t$ then $X$ is $t$-uniform;
    \item if $s \neq t$ then for any minimal filling diagram $(D,f)$ for $\gamma$, any interior vertex $v$ of $D$ is either an $s$-vertex or a $t$-vertex; moreover, 
    \begin{enumerate}
        \item if $v$ is an $s$-vertex, its neighbours are either $t$-vertices or boundary vertices,
        \item if $v$ is a $t$-vertex, its neighbours  $v_{i}$, $1 \leq i \leq t$, form a full cycle $(v_1, ..., v_t)$ such that the vertices at odd indices are $t$-vertices or boundary vertices, and the vertices at even indices are $s$-vertices or boundary vertices.
    \end{enumerate}
\end{enumerate}
\end{definition}

For $s = t$, the $s$-vertices coincide with the $t$-vertices.
Then the difference between an $s$-vertex and a $t$-vertex consists only in the role it plays.
Namely, such vertex is either an interior or a boundary vertex of a polygon.

If $t$ is even, we can construct a surface tiling using $s$-polygons.
Namely, we start with an $s$-polygon and we connect $t/2$ such polygons at each vertex.
If $t \ge 6$, one can continue this process for all boundary vertices for as long as we want.
The center vertices of the $s$-polygons have $s$ neighbours.
So they are $s$-vertices.
Because the vertices at the intersection of $t/2$ $s$-polygons have $2 \cdot t/2 = t$ triangles around them, such vertices are $t$-vertices.
There is no triangle in this construction such that all its vertices are $t$-vertices.
A minimal filling diagram is, by definition, a section of such tiling.
Note that such tiling is unique.

If $t$ is odd, each $t$-vertex may have a single pair of adjacent neighbours which are also $t$-vertices.
We can construct a surface tiling as above except that in this case we can connect only $(t-1)/2$ such polygons at each vertex.
This implies that at each $t$-vertex there is an extra triangle.
All vertices of this extra triangle are $t$-vertices (i.e. the extra triangle is not included in an $s$-polygon).
Because these extra triangles can be connected in several ways to the $s$-polygons (either to a vertex or to an edge), there are several ways to construct such tilings.
Only in case $s = t$, we obtain the same tiling no matter how we consider the extra triangles.

Based on the explanations above, we continue the analysis only for the two cases with unique tiling.
The first case is when $t$ is even; the second case is when $s = t$ odd.

Let $X$ be a $(s,t)$-uniform simplicial complex and let $\gamma$ be a loop in $X$.
Let $(D,f)$ be a minimal filling diagram for $\gamma$.
Let $v$ be an interior vertex of $D$.
We call the \textit{sphere} centered at $v$ of radius $n$ the set of edges spanned by the vertices at distance $n$ from $v$, $n \ge 0$.
We denote it by $S_n^{s,t}$.
We call the \textit{area} of a sphere the number of triangles inside the sphere.
We denote it by $A_{n}^{s,t}$.
We call the \textit{length} of a sphere the number of edges on the sphere.
We denote it by $|S_n^{s,t}|$.

If $Z$ is a set of simplices, we denote by $|Z|$ the number of simplices in the set.

\subsection{Homogeneous linear recurrence relations}

A homogeneous linear recurrence relation of order $d$ with constant coefficients $c_1$, $c_2$, $\dots$, $c_d$ is an equation of the form
$$x_n = c_1 x_{n-1} + c_2 x_{n-2} + \dots + c_d x_{n-d}$$

A constant-recursive sequence is a sequence satisfying a recurrence relation of this form.
The initial values $x_0$, $\dots$, $x_{d-1}$ can be taken to be any values but then the recurrence relation determines the sequence uniquely.

The characteristic equation of the recurrence relation for the constant-recursive sequence is
$$x^d - c_1 x^{d-1} - \dots - c_{d-1} x  - c_d = 0$$

If the roots $r_1$, $r_2$, $\dots$, $r_d$ of the characteristic equation are distinct, then each solution of the recurrence relation has the form
\begin{equation} \label{eq:rec-rel-general-term}
    x_n = \sum_{i=1}^{d} k_i r_i^n = k_1 r_1^n + k_2 r_2^n + \dots + k_d r_d^n
\end{equation}
The coefficients $k_i$ are determined in order to fit the initial conditions of the recurrence relation.
If the same roots occur multiple times, the terms in the formula above corresponding to the second and the later occurrences of the same root are multiplied by increasing powers of n.
Assume that there are $e$ distinct roots and that each of these roots, say $r_{i}$, occurs $p_i$ times (i.e. $\sum_{i=1}^{e} p_i = d$).
Then we have
\begin{equation}\label{eq:rec-rel-general-term-duplicate}
    x_n = \sum_{i=1}^{e} \sum_{j=1}^{p_i} k_{ij} \cdot n^{j-1} \cdot r_i^n
    = (k_{11} + k_{12} n + \dots + k_{1p_{1}} n^{p_1 - 1}) \cdot r_1^n + \dots
\end{equation}

\subsection{Combinatorial curvature}

According to Descartes' theorem, the total curvature of a polyhedron is $4\pi$.
The curvature at a vertex is the angle defect of the vertex.
At edges and faces, there is zero curvature.

We study the curvature inside loops of a simplicial complex $X$.
Let $\gamma$ be a homotopically trivial loop of $X$.
We consider a minimal filling diagram $(D,f)$ for $\gamma$.
We compute the curvature for the triangulated surface $D$.
Similar to polyhedral curvature, we measure curvature only at vertices.

If a vertex has $6$ neighbours, then it has no angle defect and no angle excess; it has zero curvature.
If a vertex has less than $6$ neighbours, then it has an angle defect; it has positive curvature.
If a vertex has more than $6$ neighbours, then it has an angle excess; it has negative curvature.
For simplicity, we consider the curvature of a $t$-vertex to be $6-t$, instead at $(6-t) \cdot \cfrac{\pi}{3}$.

Below we give the combinatorial version of the Gauss-Bonnet Theorem.

\begin{theorem}[Combinatorial Gauss-Bonnet]
Let S be a compact triangulated surface and let $v$ be an interior vertex of $S$.
Let $\chi(S)$ denote the Euler characteristic of $S$.
Let $\chi(v)$ denote the number of triangles containing the vertex $v$.
Then the following formula holds
\begin{equation}\label{eq:gauss-bonnet}
    6 \chi(S) = \sum_{v \in \partial S} (3 - \chi(v)) + \sum_{v \in int S} (6 - \chi(v))
\end{equation}
\end{theorem}

The first term of the right side of (\ref{eq:gauss-bonnet}) is called the \emph{geodesic curvature} (i.e. the curvature along a curve).
It is denoted by $k_g$
$$k_g(\partial S) = \sum_{v \in \partial S} k_g(v) = \sum_{v \in \partial S} (3 - \chi(v))$$
The second term of the right side of (\ref{eq:gauss-bonnet}) is called the \emph{Gaussian curvature} (i.e. the curvature inside a curve).
It is denoted by $K$
$$K(S) = \sum_{v \in int S} K(v) = \sum_{v \in int S} (6 - \chi(v))$$
So formula (\ref{eq:gauss-bonnet}) becomes
\begin{equation}\label{eq:gauss-bonnet-short}
    6 \chi(S) = k_g(\partial S) + K(S)
\end{equation}
Because a disc $D$ has the Euler characteristic $\chi(D) = 1$, for $S = D$ relation (\ref{eq:gauss-bonnet-short}) becomes
\begin{equation}\label{eq:gauss-bonnet-short-disk}
    6 = k_g(\partial D) + K(D)
\end{equation}
The formula above also holds for any non self-intersecting loop inside $D$.

\section{Constant-recursive sequences in (s,t)-uniform simplicial complexes}
 
Let $X$ be an $(s,t)$-uniform simplicial complex $X$.
In general we consider $s, t \ge 6$, $t$ even, except for $s = t = 6$.
In section \ref{sec:examples} we give examples of $(s,t)$-uniform simplicial complexes.
In section \ref{sec:recurrence-s-t-unif-connection} we find a connection between the lengths of spheres in $(s,t)$-uniform simplicial complexes and the terms of certain constant-recursive sequences (Theorem \ref{theorem:sphere-recurrence}).
We take into account spheres $S_n^{s,t}$ belonging to the disc of a minimal filling diagram associated to a loop of $X$.
In section \ref{sec:recurrences-s-t-unif} we analyze concrete constant-recursive sequences for $(s,t)$-uniform simplicial complexes.
Besides we find the limit of the ratio $\cfrac{A_n^{s,t}}{|S_n^{s,t}|}$ as $n$ goes to infinity (Theorem \ref{theorem:area-length-ratio-s-t-unif}).
In section \ref{sec:curvature-inside-spheres} we study the average Gaussian curvature for vertices inside spheres as the radii of spheres grow.

\subsection{Examples}\label{sec:examples}

We start by giving a few examples of $(s,t)$-uniform simplicial complexes.
We present two diagrams for each such complex.
In each diagram spheres are centered at a vertex $v$.
On the left side $v$ is an $s$-vertex and we denote a sphere by $S_n^{s,t}$.
On the right side $v$ is a $t$-vertex and we denote a sphere by $T_n^{s,t}$.
In each case we represent spheres up to radius $4$.
In the rest of the paper we refer to spheres as $S_n^{s,t}$, no matter the type of vertex the sphere is centered at.

\begin{figure}[ht]
    \centering
    \begin{subfigure}[b]{.5\textwidth}
        \centering
        \includegraphics[width=.95\textwidth]{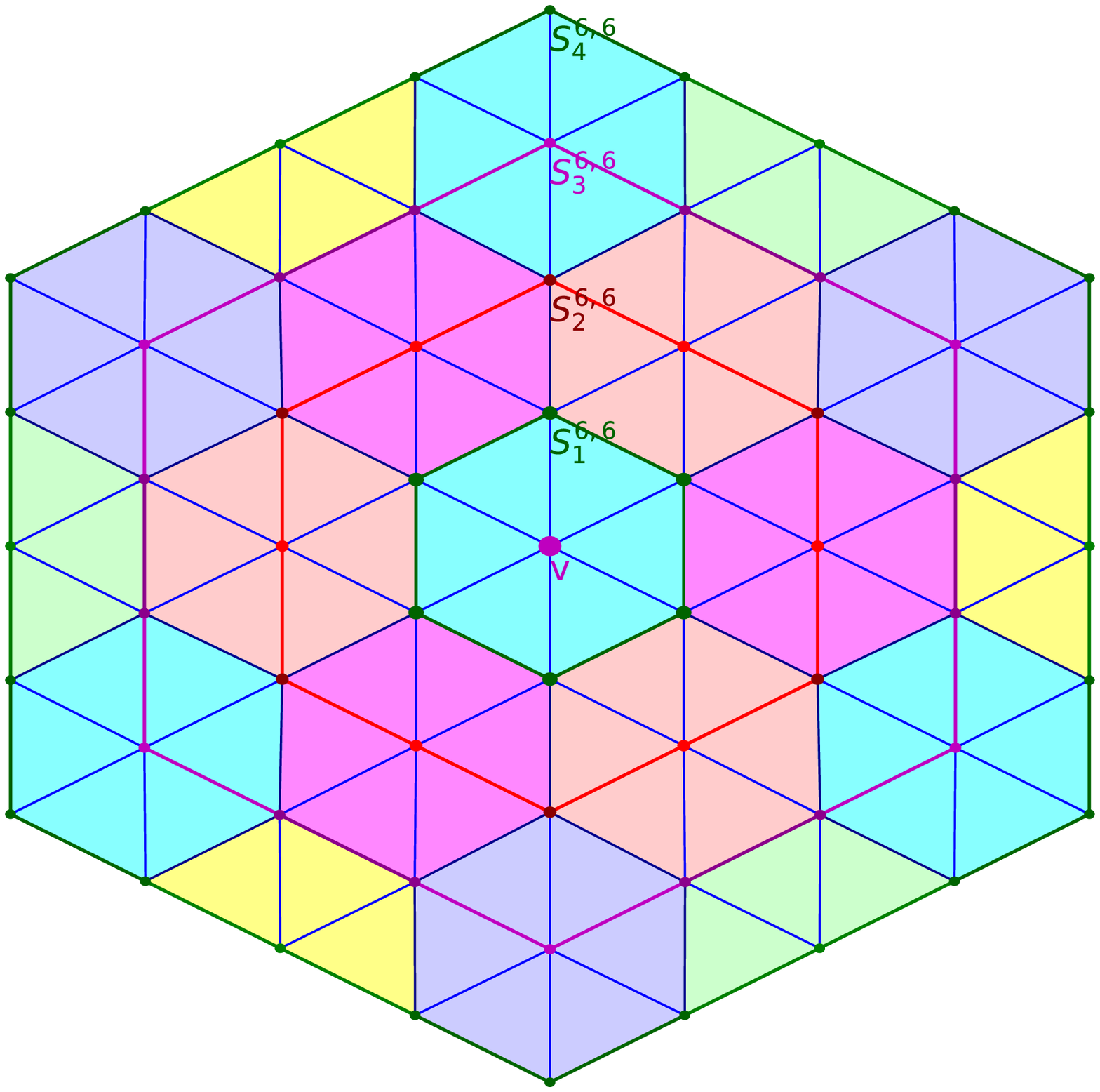}
    \end{subfigure}%
    \begin{subfigure}[b]{.5\textwidth}
        \centering
        \includegraphics[width=.95\textwidth]{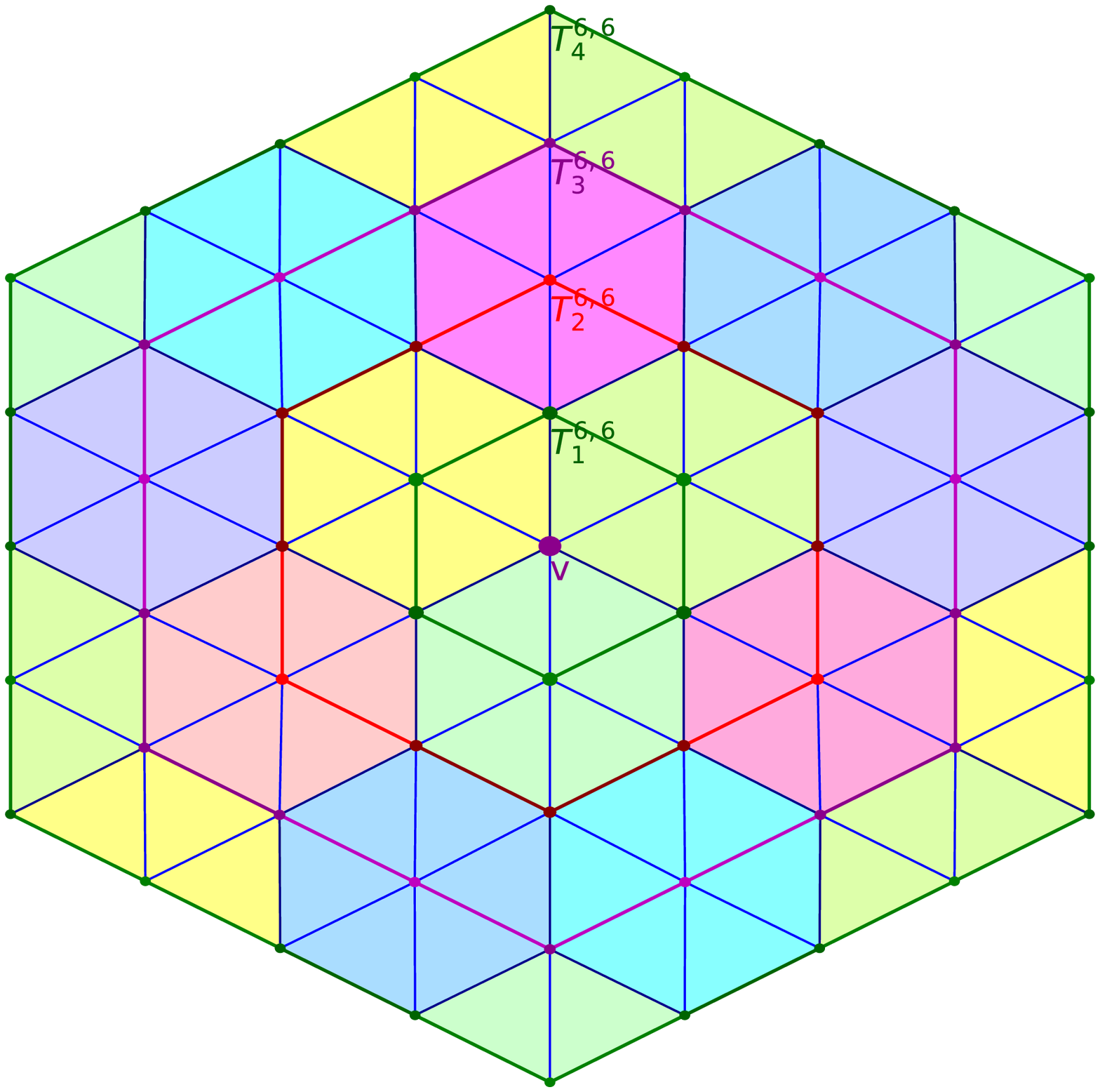}
    \end{subfigure}%
    \caption{Diagrams of a $(6,6)$-uniform simplicial complex}
    \label{fig:6-6-unif}
\end{figure}

We color the interior of each $s$-polygon with  different colors.
This outlines the type of vertex and edge on each sphere.
The parameter $t$ is even in all examples with the exception of Figure \ref{fig:7-7-unif} where $s=t=7$.
This is an example of $7$-uniform complex.
In Figure \ref{fig:7-7-unif} we color in gray the triangles that do not belong to any $s$-polygon at all.

\begin{figure}[p]
    \centering
    \begin{subfigure}[b]{.5\textwidth}
        \centering
        \includegraphics[width=.95\textwidth]{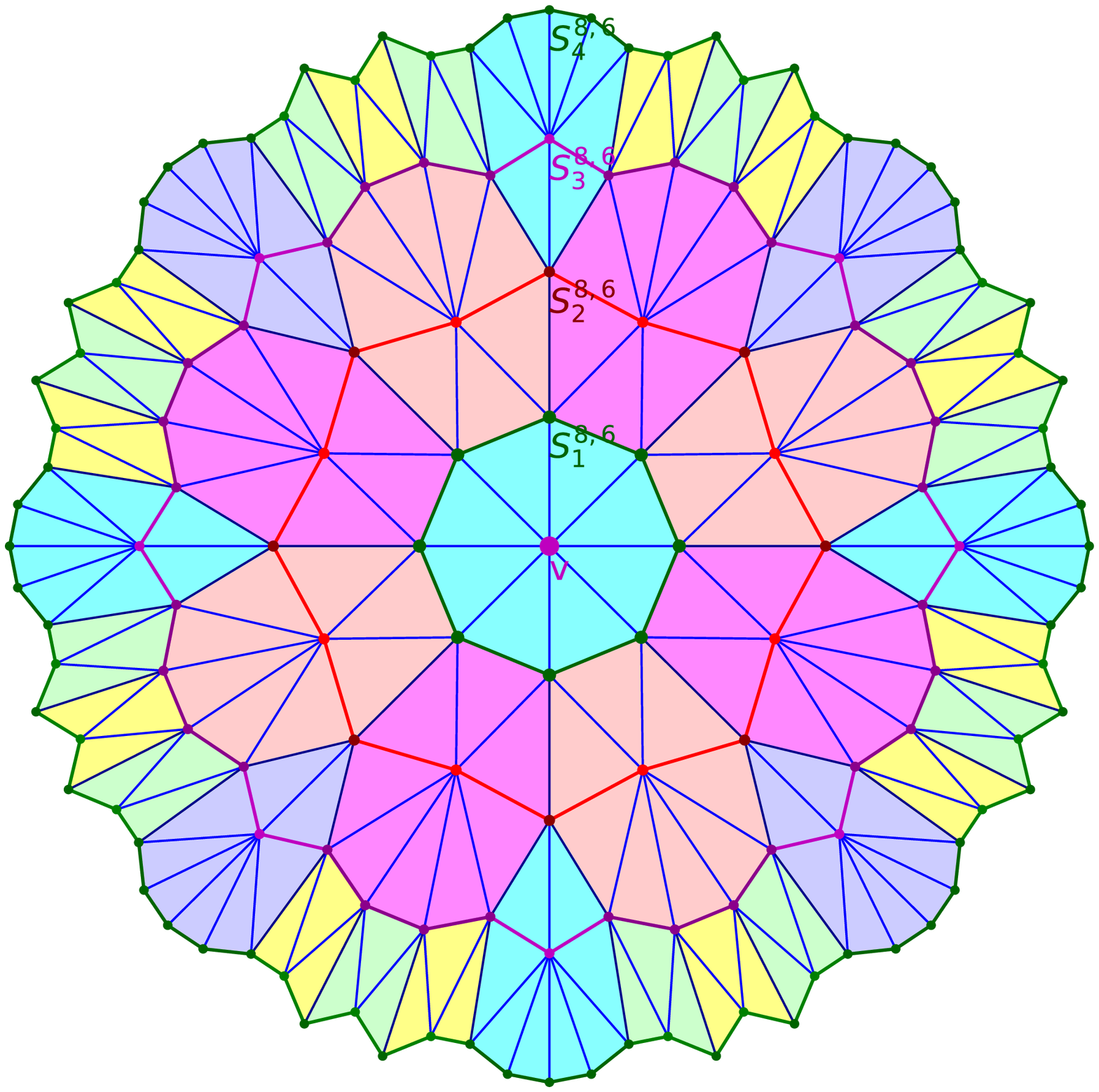}
    \end{subfigure}%
    \begin{subfigure}[b]{.5\textwidth}
        \centering
        \includegraphics[width=.95\textwidth]{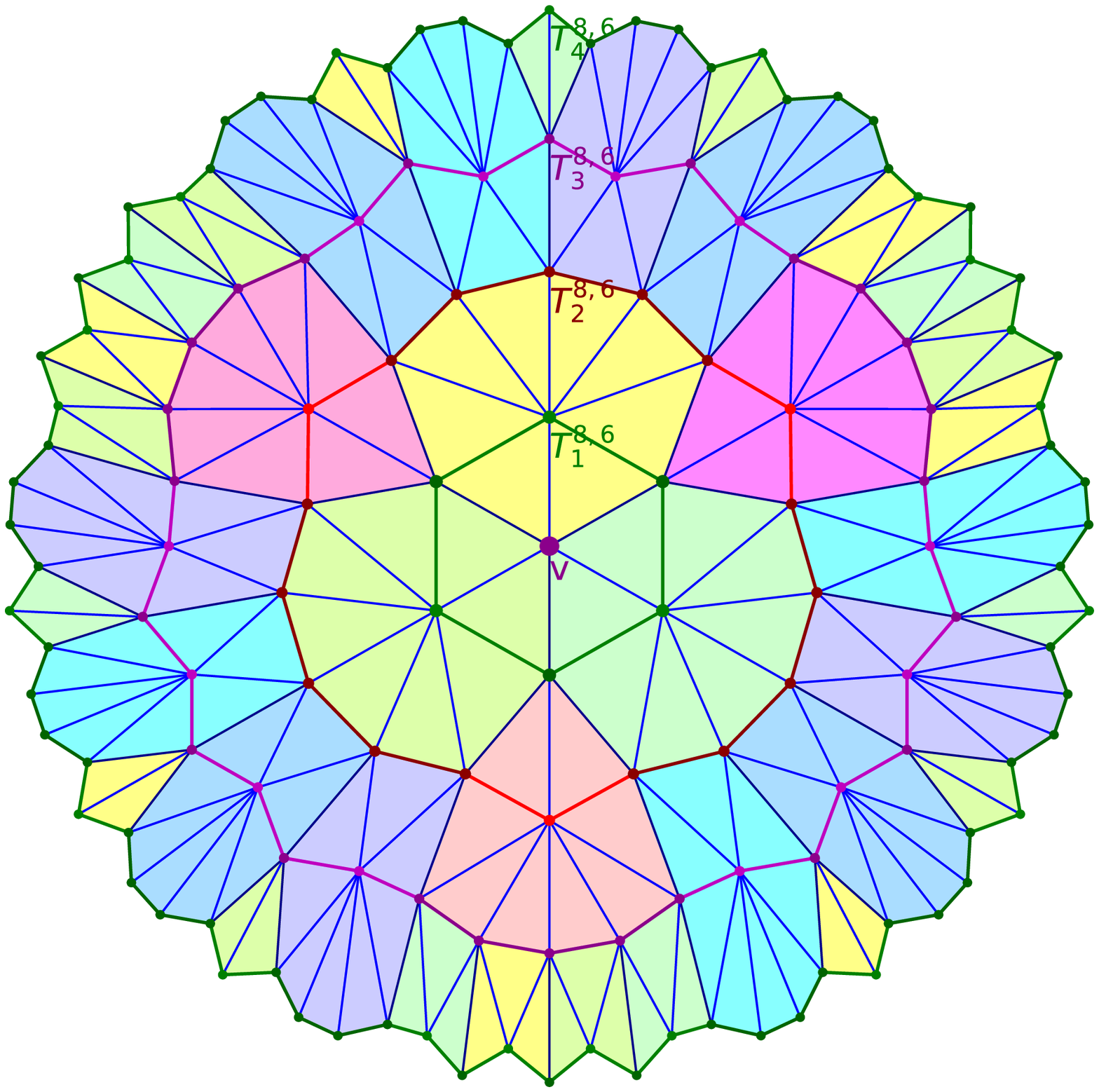}
    \end{subfigure}%
    \caption{Diagrams of an $(8,6)$-uniform simplicial complex}
    \label{fig:8-6-unif}

    \centering
    \begin{subfigure}[b]{.5\textwidth}
        \centering
        \includegraphics[width=.95\textwidth]{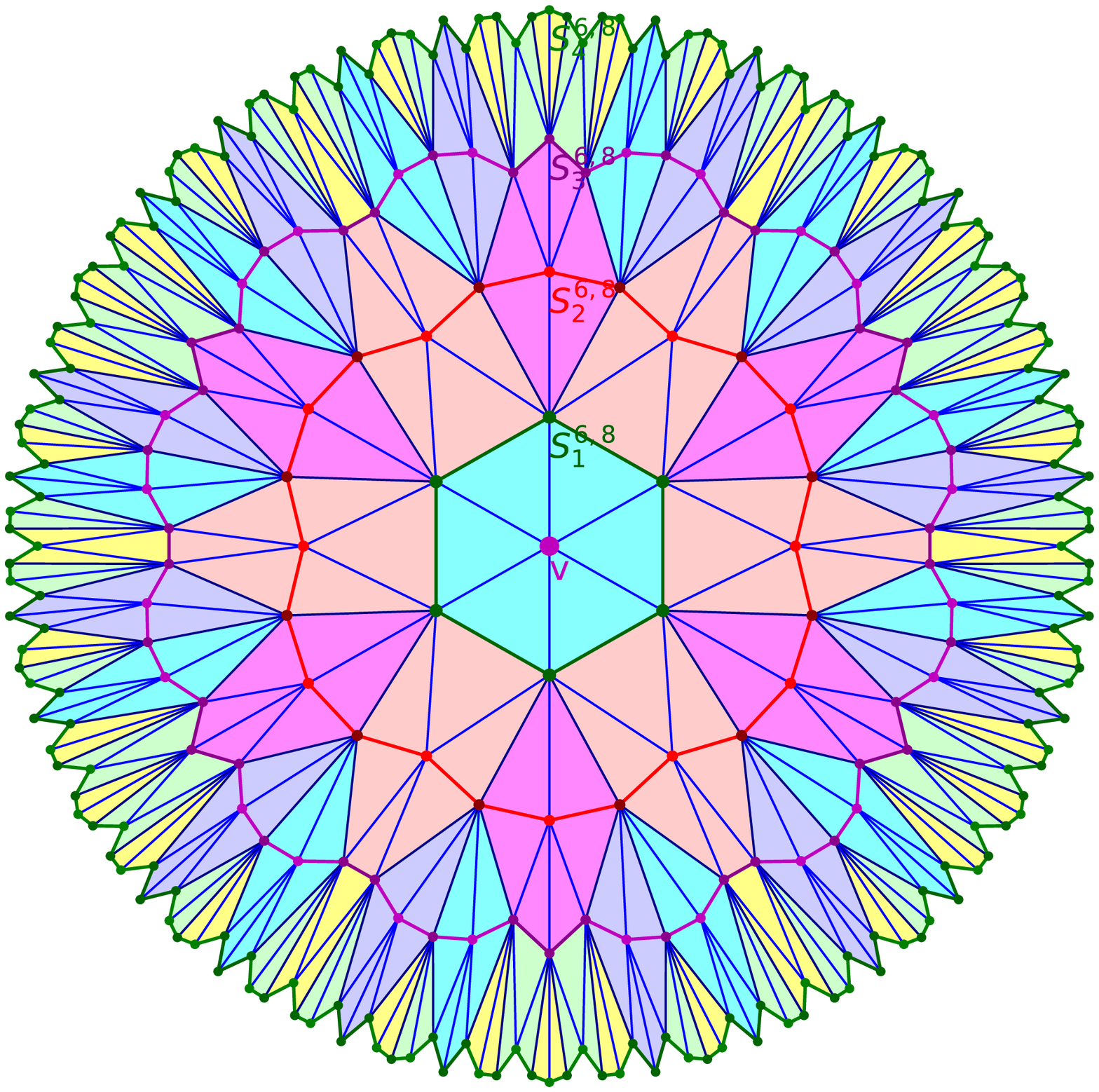}
    \end{subfigure}%
    \begin{subfigure}[b]{.5\textwidth}
        \centering
        \includegraphics[width=.95\textwidth]{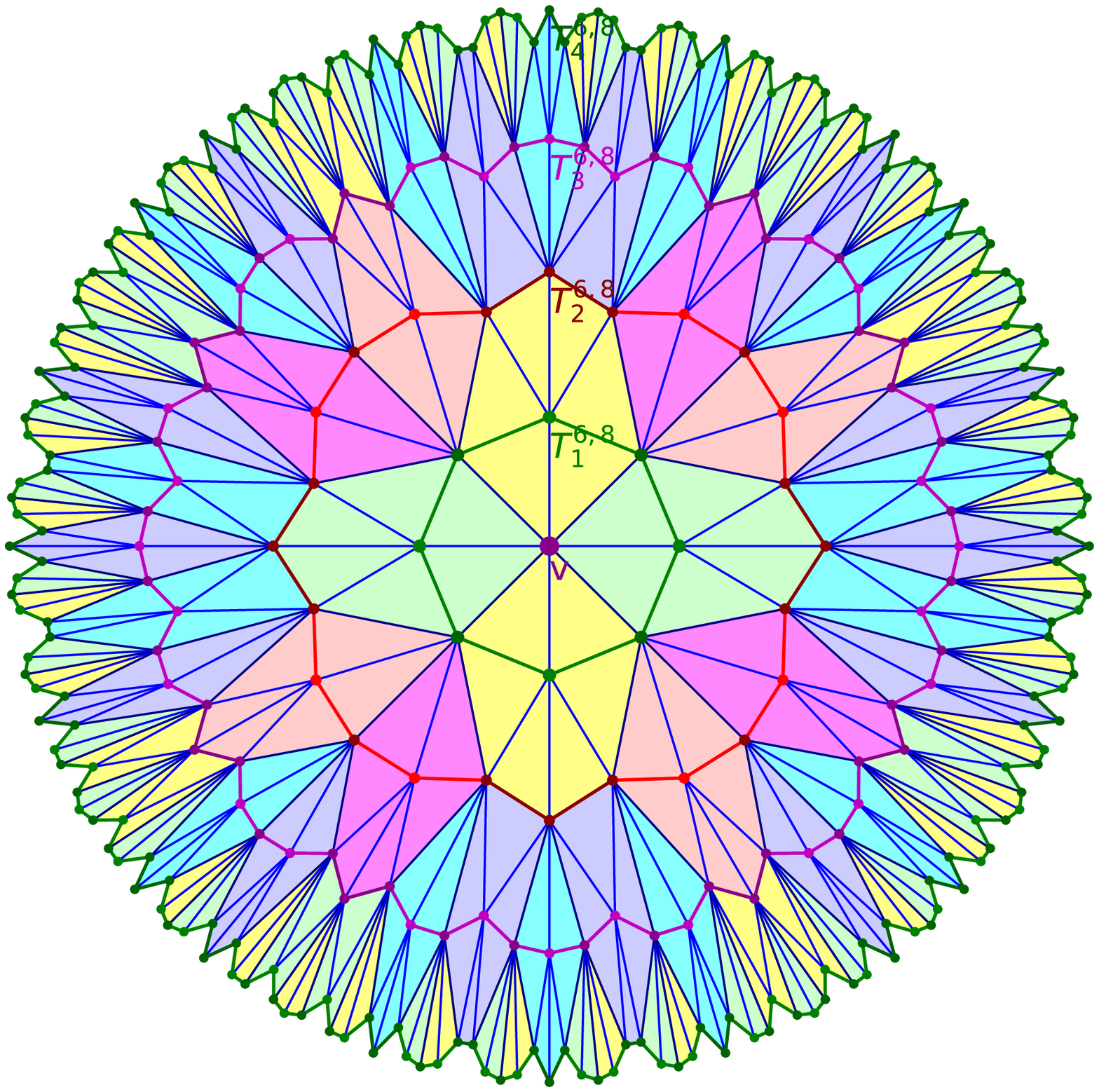}
    \end{subfigure}%
    \caption{Diagrams of a $(6,8)$-uniform simplicial complex}
    \label{fig:6-8-unif}

    \centering
    \begin{subfigure}[b]{.5\textwidth}
        \centering
        \includegraphics[width=.95\textwidth]{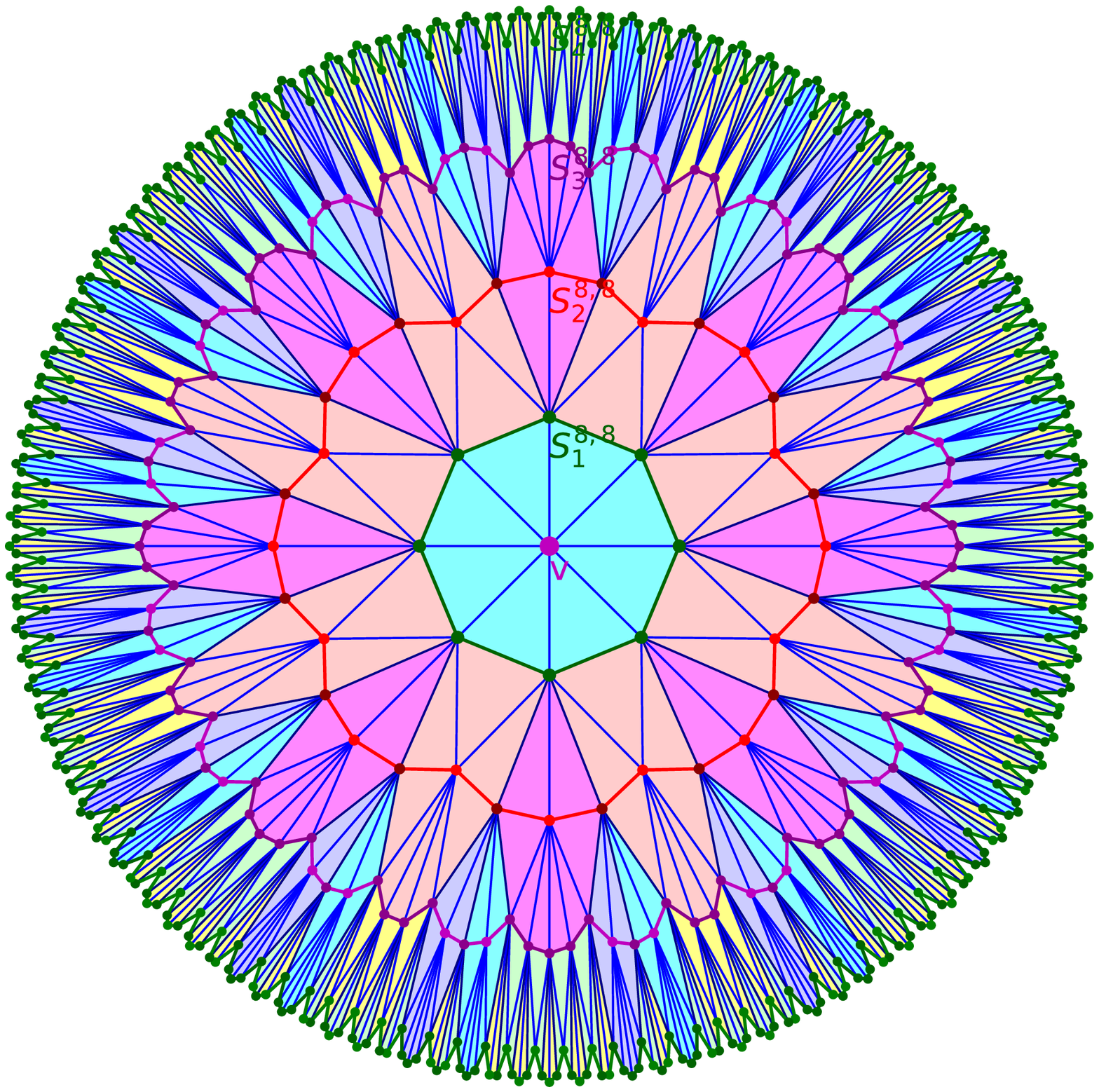}
    \end{subfigure}%
    \begin{subfigure}[b]{.5\textwidth}
        \centering
        \includegraphics[width=.95\textwidth]{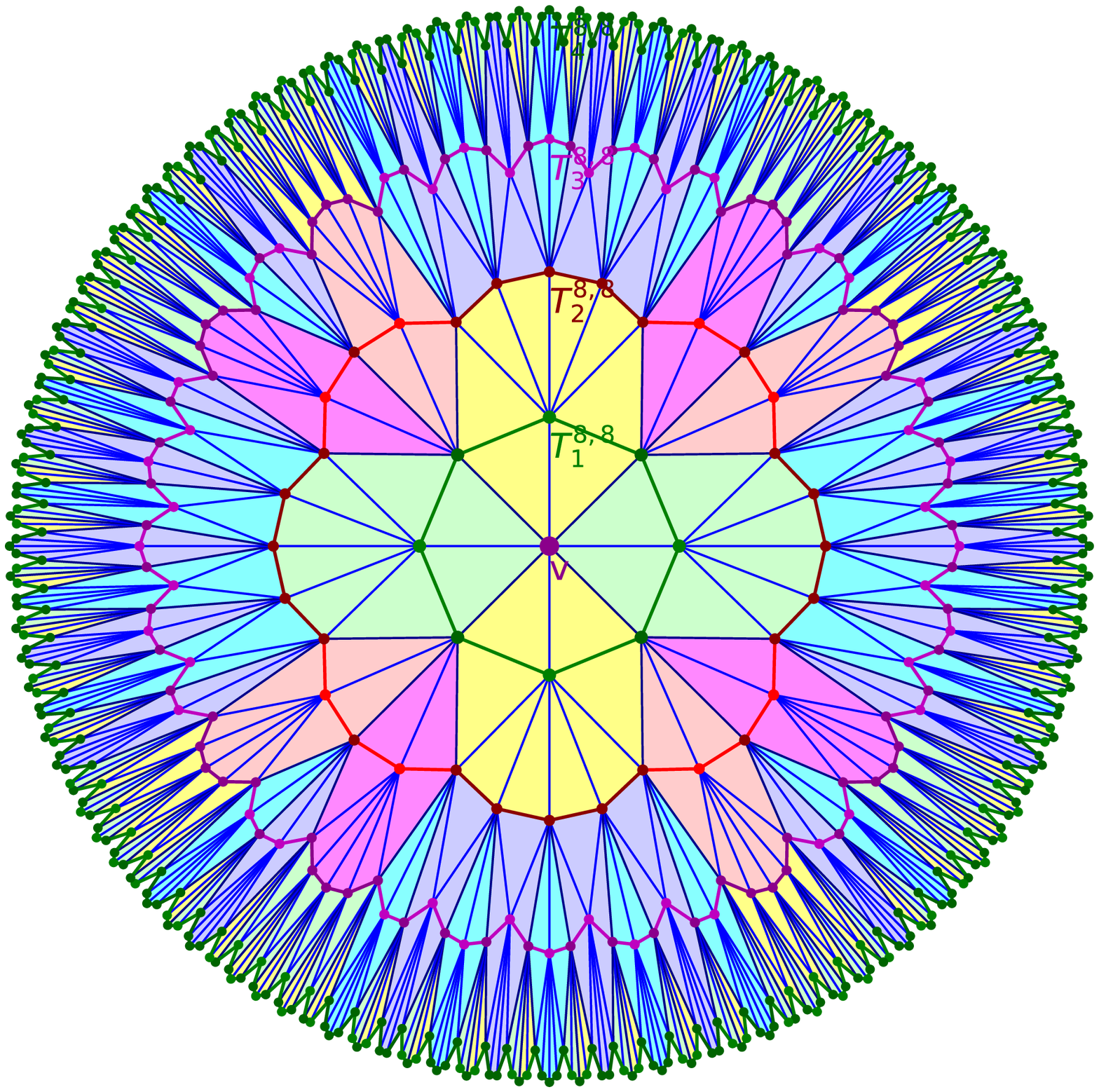}
    \end{subfigure}%
    \caption{Diagrams of an $(8,8)$-uniform simplicial complex}
    \label{fig:8-8-unif}
\end{figure}

\begin{figure}[p]
    \centering
    \begin{subfigure}[b]{.5\textwidth}
        \centering
        \includegraphics[width=.95\textwidth]{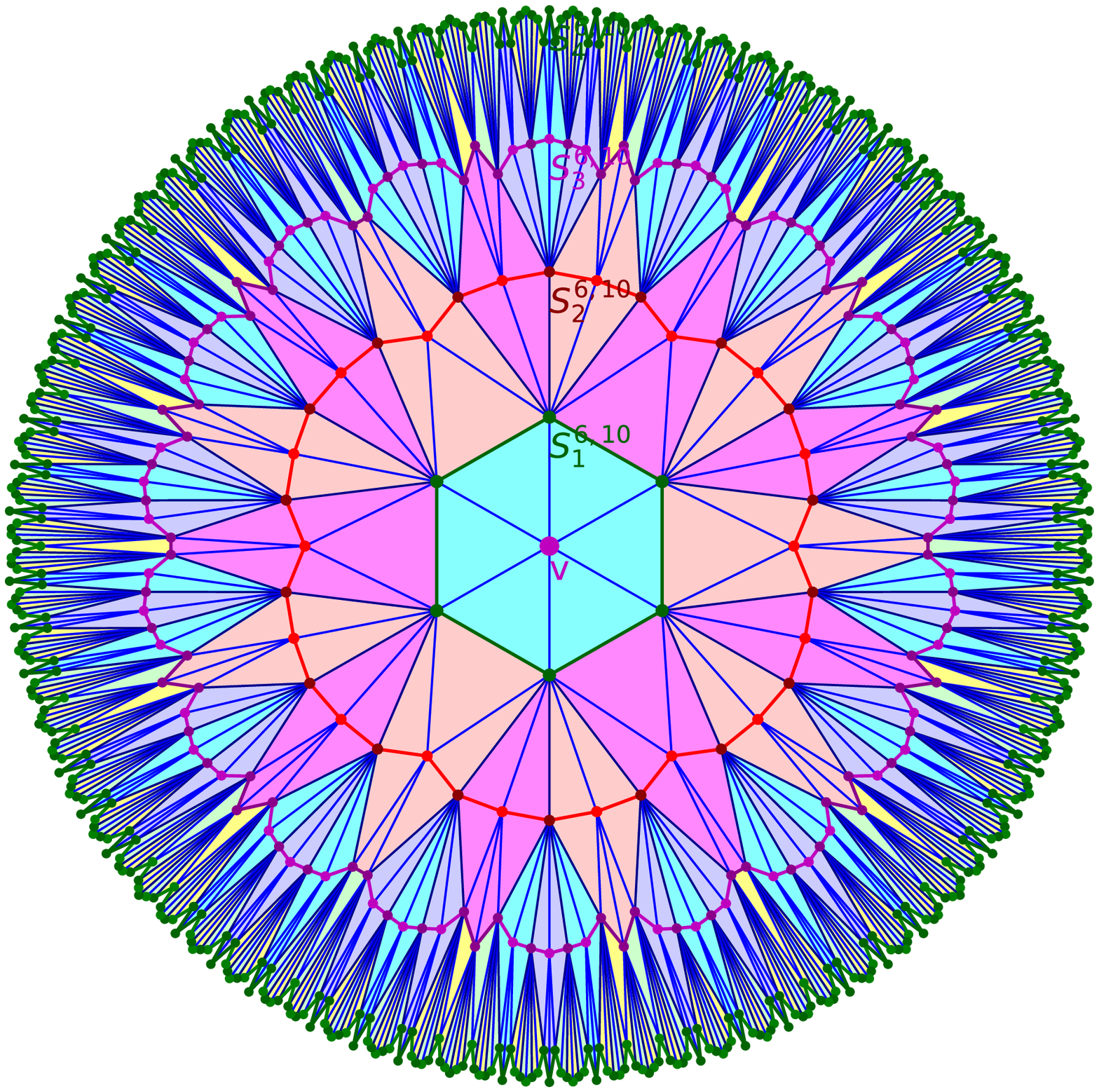}
    \end{subfigure}%
    \begin{subfigure}[b]{.5\textwidth}
        \centering
        \includegraphics[width=.95\textwidth]{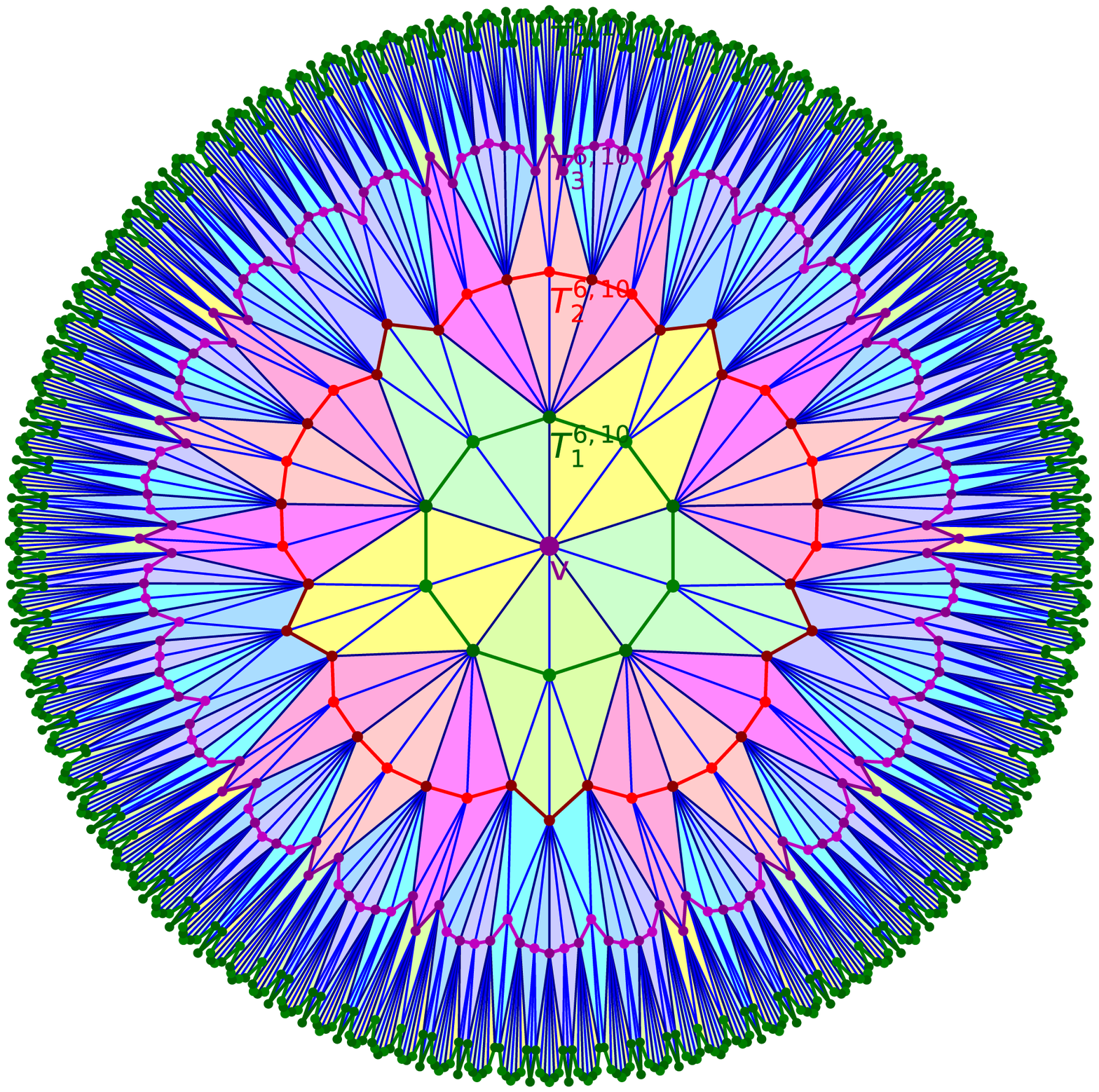}
    \end{subfigure}%
    \caption{Diagrams of a $(6,10)$-uniform simplicial complex}
    \label{fig:6-10-unif}

    \centering
    \begin{subfigure}[b]{.5\textwidth}
        \centering
        \includegraphics[width=.95\textwidth]{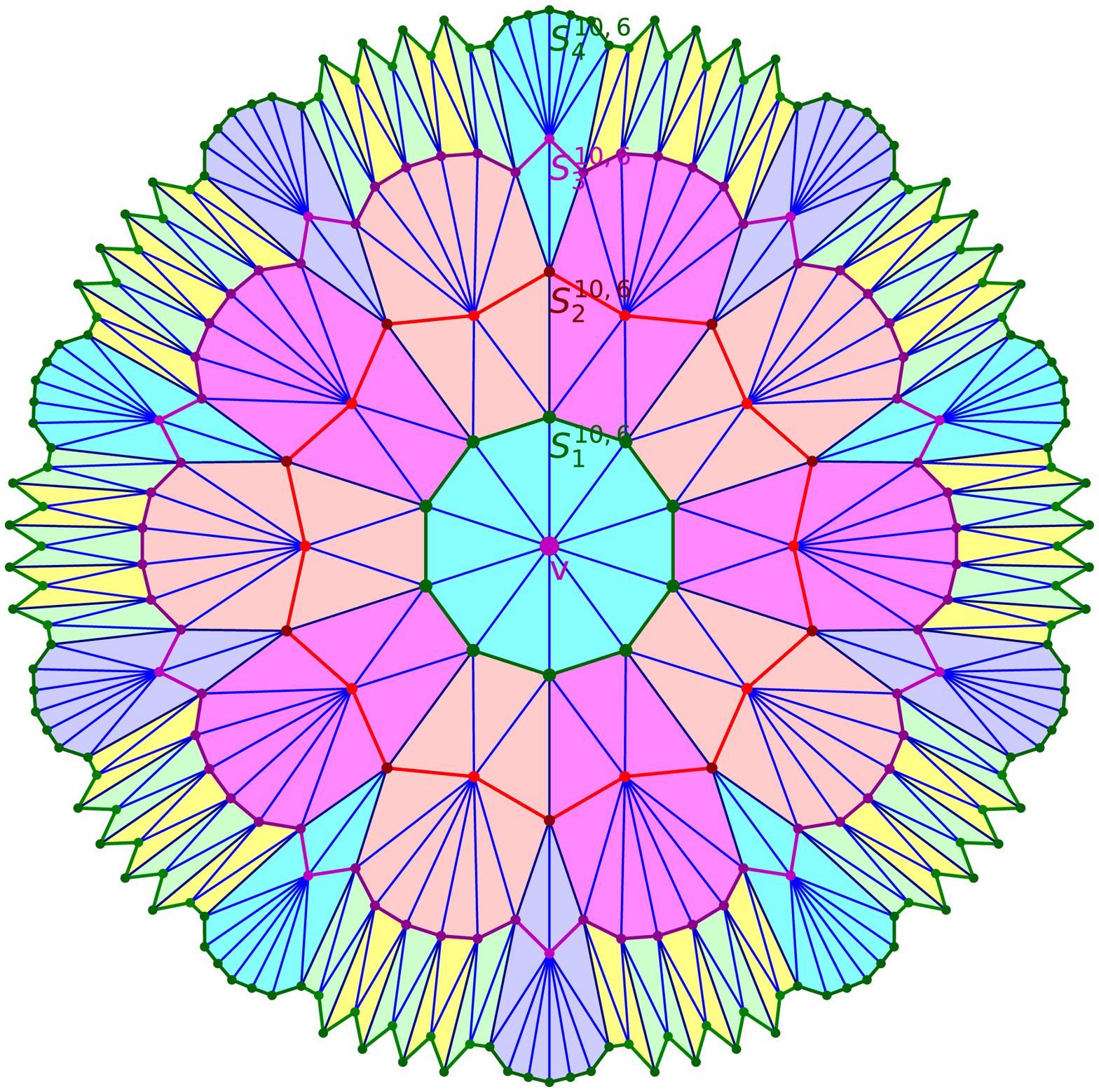}
    \end{subfigure}%
    \begin{subfigure}[b]{.5\textwidth}
        \centering
        \includegraphics[width=.95\textwidth]{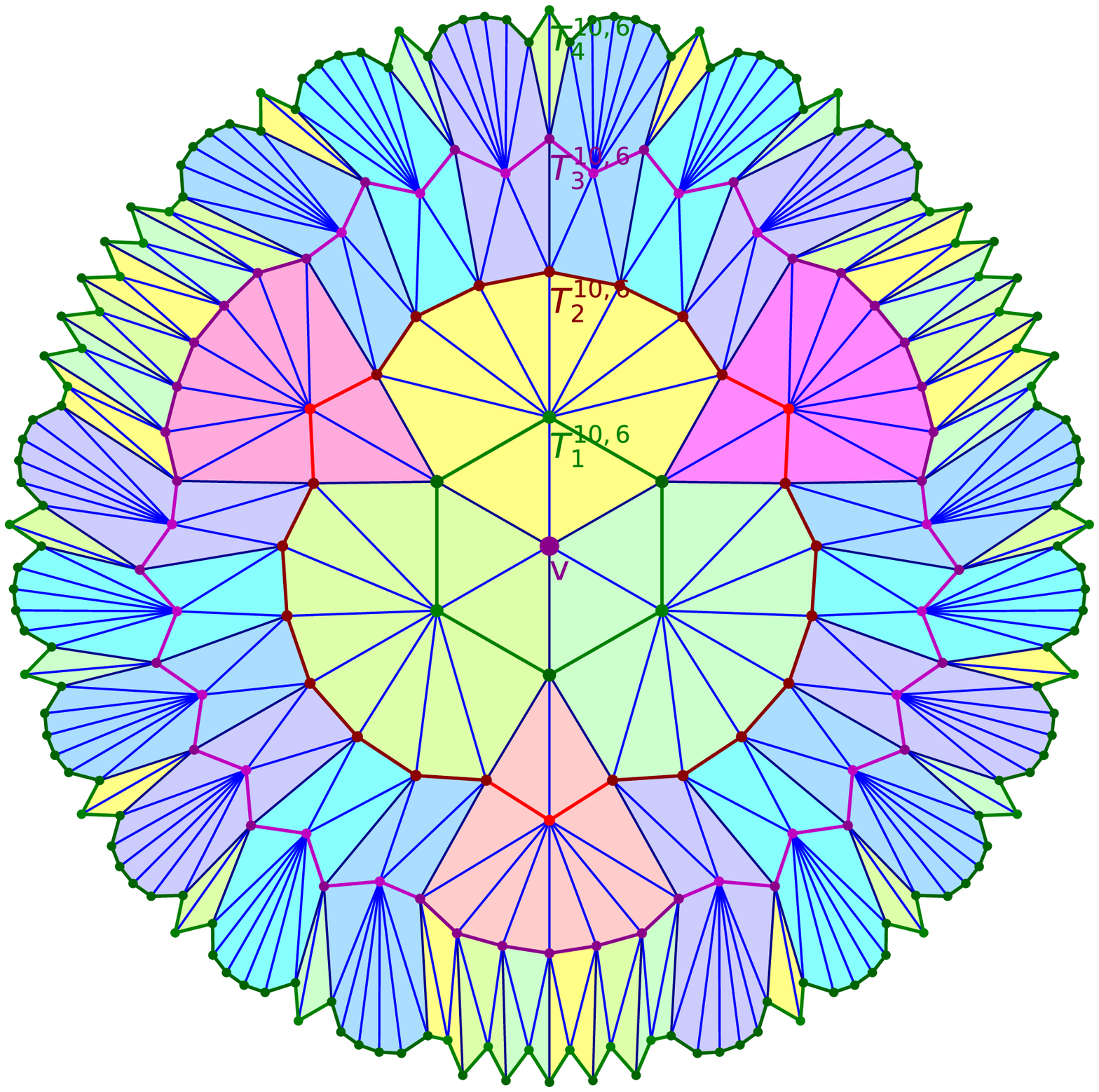}
    \end{subfigure}%
    \caption{Diagrams of a $(10,6)$-uniform simplicial complex}
    \label{fig:10-6-unif}

    \centering
    \begin{subfigure}[b]{.5\textwidth}
        \centering
        \includegraphics[width=.95\textwidth]{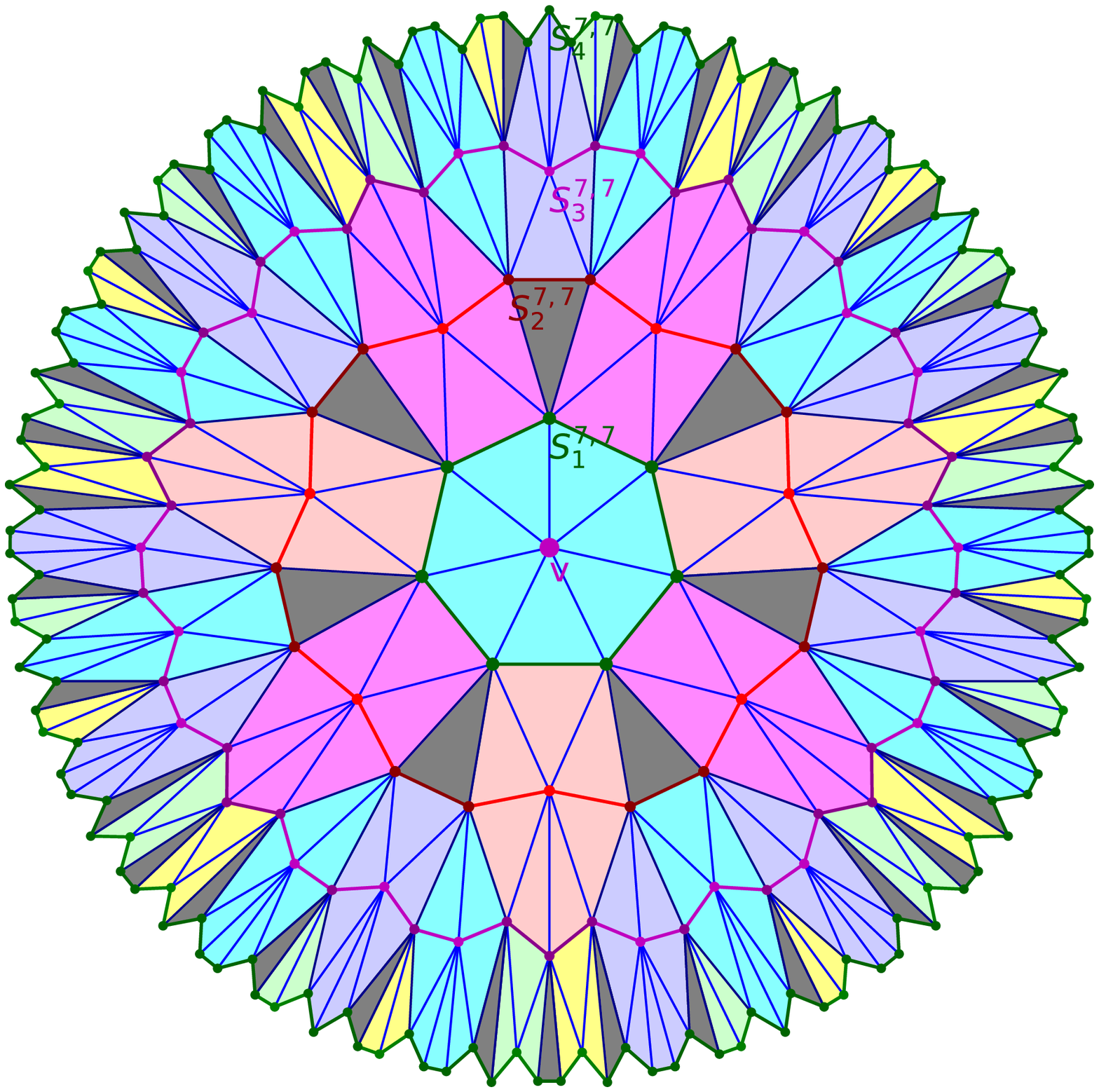}
    \end{subfigure}%
    \begin{subfigure}[b]{.5\textwidth}
        \centering
        \includegraphics[width=.95\textwidth]{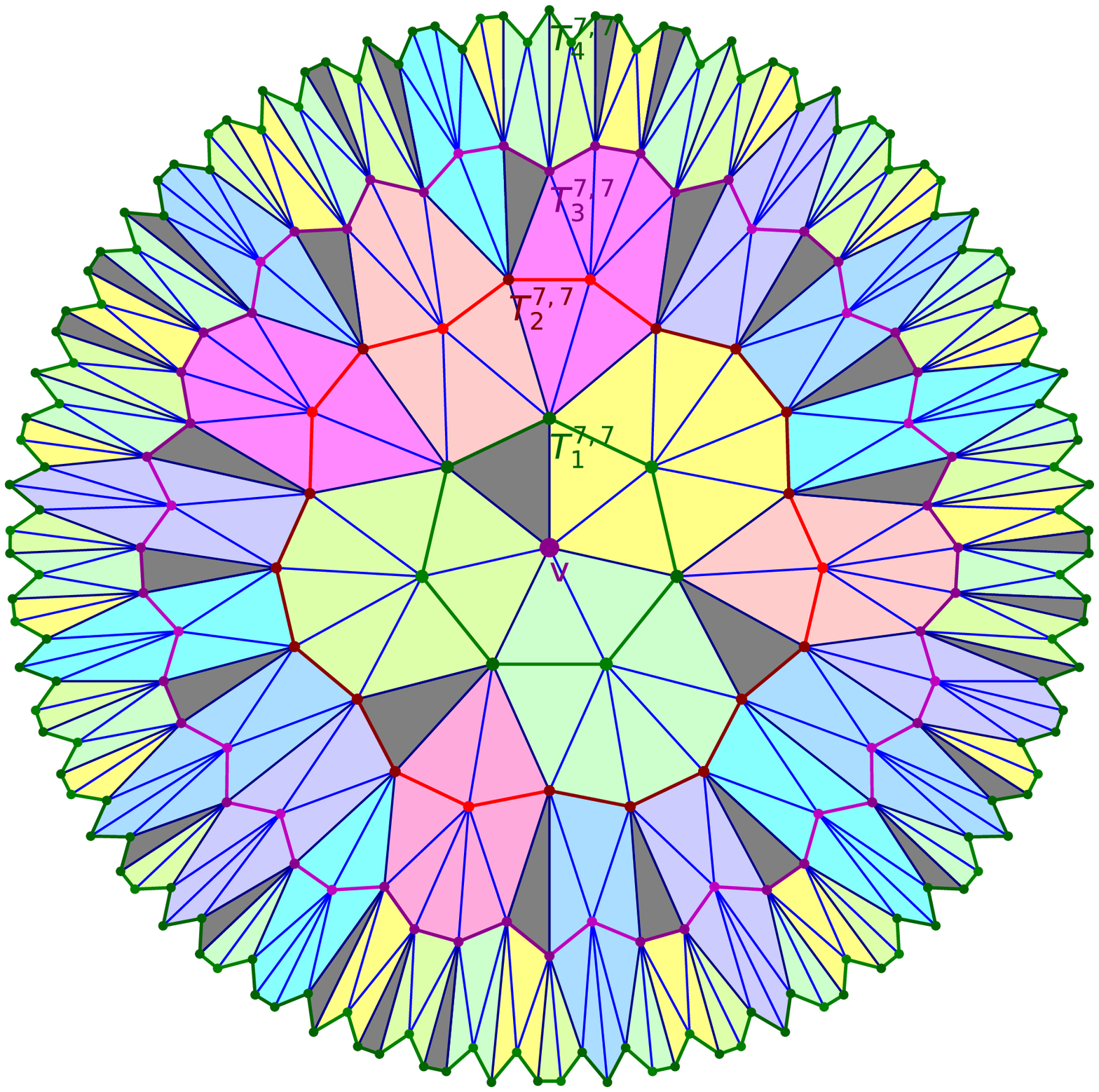}
    \end{subfigure}%
    \caption{Diagrams of a $(7,7)$-uniform simplicial complex}
    \label{fig:7-7-unif}
\end{figure}

\subsection{Recurrence relations in (s,t)-uniform simplicial complexes}\label{sec:recurrence-s-t-unif-connection}

We denote by $V_n^{s,t}$ the set of $s$-vertices on a sphere $S_n^{s,t}$.
We denote by $W_n^{s,t}$ the set of $t$-vertices on $S_n^{s,t}$.
We denote by $E_n^{s,t}$ the set of $(s,t)$-edges on $S_n^{s,t}$.
We denote by $F_n^{s,t}$ the set of $(t,t)$-edges on $S_n^{s,t}$.

We consider only vertices "connected to edges on the same sphere".
This is necessary to avoid counting the center vertex of $S_0^{s,t}$ (which is a sphere centered at a vertex).
The sphere $S_0^{s,t}$ has length $0$.
Because $S_0^{s,t}$ contains no edges, it also contains no vertices.

Some quick observations lead to the following relations:
\begin{equation}\label{eq:rec-spheres-svw}
    |S_n^{s,t}| = |V_n^{s,t}| + |W_n^{s,t}|,
\end{equation}
\begin{equation}\label{eq:rec-spheres-sef}
    |S_n^{s,t}| = |E_n^{s,t}| + |F_n^{s,t}|,
\end{equation}
\begin{equation}\label{eq:rec-spheres-ve}
    |E_n^{s,t}| = 2 \cdot |V_n^{s,t}|
\end{equation}
The last relation holds because an $s$-vertex has only $t$-vertices as neighbours.
So an $s$-vertex on a sphere is adjacent to two $t$-vertices which lie on the same sphere.
For each $s$-vertex there are therefore on a sphere two $(s,t)$-edges.

We start by establishing a connection between the lengths of spheres in $(s,t)$-uniform simplicial complexes and the terms of certain constant-recursive sequences.
Besides we find other simplices on spheres that satisfy the same recurrence relation.

\begin{theorem}\label{theorem:sphere-recurrence}
Let $X$ be an $(s,t)$-uniform simplicial complex.
Let $\gamma$ be a loop of $X$ and let $(D,f)$ be a minimal filling diagram for $\gamma$.
Let $v$ be an interior vertex of $D$.
Then the lengths of spheres centered at $v$ are sequences that satisfy a certain recurrence relation with parameters depending on $T = t-4$ and $S = s-4$.
Namely,
\begin{equation}\label{eq:rec-spheres}
    |S_n^{s,t}| = \cfrac{T}{2} \cdot |S_{n-1}^{s,t}| + \bigg(\cfrac{S \cdot T}{2}-2\bigg) \cdot |S_{n-2}^{s,t}| + \cfrac{T}{2} \cdot |S_{n-3}^{s,t}| - |S_{n-4}^{s,t}|
\end{equation}
Moreover, the number of the following simplices satisfy the same recurrence relation:
\begin{itemize}
    \item the $s$-vertices ($|V_n^{s,t}|$),
    \item the $t$-vertices ($|W_n^{s,t}|$),
    \item the $(s,t)$-edges ($|E_n^{s,t}|$),
    \item the $(t,t)$-edges ($|F_n^{s,t}|$).
\end{itemize}
\end{theorem}

\begin{proof}
We start by proving this relation for $n \ge 2$
\begin{equation}\label{eq:rec-spheres-rel1}
    |S_n^{s,t}| = S \cdot |V_{n-1}^{s,t}| + T \cdot |W_{n-1}^{s,t}| - |S_{n-2}^{s,t}|
\end{equation}

The proof of this relation is similar to the one given in \cite{L-t-unif} (Theorem $3.1$).
We consider consecutive spheres of $D$ centered at the same vertex $v$.
We count the number of vertices on each sphere.
In general the length of a sphere $S_n^{s,t}$ equals the number of vertices on the sphere.
These vertices are at distance $n$ from the center of the sphere.
The sphere $S_0^{s,t}$, however, is an exception.
Although  the length of this sphere is $0$, the number of its vertices is equal to $1$.

Let $n \ge 2$.
In order to count the number of vertices on the sphere $S_n^{s,t}$, we split the vertices on $S_{n-1}^{s,t}$ into two sets:
\begin{itemize}
    \item the set $Y_{n-1}$ contains those vertices adjacent to two vertices on $S_{n-2}^{s,t}$ which are adjacent; let $Y_{n-1}^V = Y_{n-1} \cap V_{n-1}^{s,t}$ and let $Y_{n-1}^W = Y_{n-1} \cap W_{n-1}^{s,t}$;
    \item the set $Z_{n-1}$ contains those vertices adjacent to one vertex on $S_{n-2}^{s,t}$; let $Z_{n-1}^V = Z_{n-1} \cap V_{n-1}^{s,t}$ and let $Z_{n-1}^W = Z_{n-1} \cap W_{n-1}^{s,t}$.
\end{itemize}

The vertices from the set $Y_{n-1}$ are adjacent to two vertices on $S_{n-2}^{s,t}$, to two vertices on the same sphere $S_{n-1}^{s,t}$, and, because the complex is $(s,t)$-uniform, to other $S = s - 4$ or $T = t - 4$ vertices on the exterior sphere $S_n^{s,t}$.
The number of these vertices is equal to the number of edges on the sphere $S_{n-2}^{s,t}$.
We note that each vertex in $Y_{n-1}$ corresponds to an edge on $S_{n-2}^{s,t}$.
So $|Y_{n-1}| = |S_{n-2}^{s,t}|$.
The vertices from the set $Z_{n-1}$ are adjacent to one vertex on $S_{n-2}^{s,t}$, to two vertices on the same sphere $S_{n-1}^{s,t}$, and, because the complex is $(s,t)$-uniform, to other $S + 1 = s - 3$ or $T + 1 = t - 3$ vertices on the exterior sphere $S_n^{s,t}$.
Any two vertices spanning an edge on $S_{n-1}^{s,t}$ are adjacent to the same vertex on $S_n^{s,t}$.
Therefore, in order to obtain the number of vertices on $S_n^{s,t}$, we have to count one vertex less for each vertex on $S_{n-1}^{s,t}$.

In conclusion, for $n \ge 2$, the number of vertices on $S_n^{s,t}$ is equal to
$$|S_n^{s,t}| = (S - 1) \cdot |Y_{n-1}^V| + [(S + 1) - 1] \cdot |Z_{n-1}^V| +
(T - 1) \cdot |Y_{n-1}^W| + [(T + 1) - 1] \cdot |Z_{n-1}^W| =$$
$$= (S - 1) \cdot |Y_{n-1}^V| +
(T - 1) \cdot |Y_{n-1}^W| + S \cdot |Z_{n-1}^V| + T \cdot |Z_{n-1}^W| =$$
$$= S \cdot (|Y_{n-1}^V| + |Z_{n-1}^V|) +
T \cdot (|Y_{n-1}^W| + |Z_{n-1}^W|) - (|Y_{n-1}^V| + |Y_{n-1}^W|) =$$
$$= S \cdot |V_{n-1}^{s,t}| + T \cdot |W_{n-1}^{s,t}| - |Y_{n-1}| =$$
$$= S \cdot |V_{n-1}^{s,t}| + T \cdot |W_{n-1}^{s,t}| - |S_{n-2}^{s,t}|$$
This proves relation (\ref{eq:rec-spheres-rel1}).

Next, for $n \ge 2$, we prove that
\begin{equation}\label{eq:rec-spheres-rel2}
    \frac{T}{2} \cdot |W_{n-1}^{s,t}| = |V_{n}^{s,t}| + |V_{n-2}^{s,t}|
\end{equation}

The set $W_{n-1}^{s,t}$ contains the $t$-vertices on the sphere $S_{n-1}^{s,t}$.
Each $t$-vertex is adjacent to $\cfrac{t}{2} = \bigg(\cfrac{T}{2} + 2\bigg)$ $s$-vertices.
These $s$-vertices can be either on the same sphere $S_{n-1}^{s,t}$ or on $S_{n}^{s,t}$ or on $S_{n-2}^{s,t}$.
We show that we can assign to each $t$-vertex from $W_{n-1}^{s,t}$, $\bigg(\cfrac{t}{2} - 2\bigg) = \cfrac{T}{2}$ unique $s$-vertices on $S_{n}^{s,t}$ and on $S_{n-2}^{s,t}$.
This assignment is done such that all $s$-vertices on $S_{n}^{s,t}$ and on $S_{n-2}^{s,t}$ correspond to some $t$-vertex from $W_{n-1}^{s,t}$.
By unique we mean that each $s$-vertex on $S_{n}^{s,t}$ and on $S_{n-2}^{s,t}$ is assigned to a single $t$-vertex from $W_{n-1}^{s,t}$.

We start by showing that any group of adjacent $t$-vertices from $W_{n-1}^{s,t}$ is adjacent to the same $s$-vertex on $S_{n-2}^{s,t}$.
We prove this at first for a group of three $t$-vertices $t_{1}, t_{2}, t_{3}$.
Assume otherwise.
Then there is at least one $t$-vertex $w_{1}$ on $S_{n-2}^{s,t}$ adjacent to $t_{2}$ such that $v_{1} \sim w_{1} \sim v_{2}$ with $v_{1},v_{2}$ $s$-vertices on $S_{n-2}^{s,t}$.
Note that $w_{1}$ is adjacent to $t-3 \ge 3$ vertices on $S_{n-3}^{s,t}$.
Because $w_{1}$ is a $t$-vertex, one of these $t-3$ vertices, say $v_{1}$, is an $s$-vertex.
Since $v_{1}$ is adjacent to three vertices on $S_{n-3}^{s,t}$ and on $S_{n-2}^{s,t}$, it is adjacent to $s-3 \ge 3$ vertices on $S_{n-4}^{s,t}$.
By continuing this process, we note that this section of inner spheres does not get smaller.
This yields a contradiction, as the center of the spheres is a vertex.
Let $t_{1} \sim t_{2} \sim ... \sim t_{m}$, $m > 3$ be $t$-vertices from $W_{n-1}^{s,t}$.
As shown above $t_{1},t_{2},t_{3}$ are adjacent to the same $s$-vertex $v$ on $S_{n-2}^{s,t}$.
This implies that the vertices $t_{2},t_{3},t_{4}$ are also adjacent to $v$.
We repeat the argument for all $t_{i}, t_{i+1}, t_{i+2}$, $3 \leq i \leq m-2$.
So any group of adjacent $t$-vertices from $W_{n-1}^{s,t}$ is adjacent to the same $s$-vertex on $S_{n-2}^{s,t}$.

In the discussion below the $t$-vertices from $W_{n-1}^{s,t}$ are considered in the order they lie on $S_{n-1}^{s,t}$ from left to right.

Let $n = 2$.
The sphere $S_1^{s,t}$ centered at an $s$-vertex is a special case because all its vertices are $t$-vertices.
Then each such $t$-vertex is adjacent to the center $s$-vertex $v$.
Because $|S_0^{s,t}| = 0$, $v$ is not counted in $|S_0^{s,t}|$ and it is therefore not assigned to any $t$-vertex on $S_{1}^{s,t}$.
Each $t$-vertex $w$ from $W_1^{s,t}$ is adjacent to $\cfrac{t}{2}$ $s$-vertices on $S_{0}^{s,t}$, on $S_{1}^{s,t}$ and on $S_{2}^{s,t}$.
Out of these $\cfrac{t}{2}$ $s$-vertices, one is on  $S_0^{s,t}$, and two are on $S_2^{s,t}$ (denoted by $v_{1},v_{2}$ from left to right).
The $s$-vertices $v_{1},v_{2}$ are also adjacent to the neighbour $t$-vertices $w_{1},w_{2}$  of $w$ on $S_1^{s,t}$, $v_1 \sim w_1$, $v_2 \sim w_2$.
We assign $v_{1}$ to $w$.
We assign $v_{2}$ to $w_{2}$.
We also assign to $w$ its neighbor $s$-vertices on $S_2^{s,t}$ between $v_{1}$ and $v_{2}$.
So we have assigned to $w$, $\bigg(\cfrac{t}{2} - 2\bigg) = \cfrac{T}{2}$ $s$-vertices on $S_0^{s,t}$ and on $S_2^{s,t}$.

For the sphere $S_1^{s,t}$ centered at a $t$-vertex, all its $t$-vertices fit the first case below.

For $n > 2$, given a $t$-vertex $w$ from $W_{n-1}^{s,t}$, one of the following situations occurs.

\begin{enumerate}
    \item Assume $w$ is adjacent to two $s$-vertices on the same sphere $S_{n-1}^{s,t}$.
    Then $w$ is adjacent to $\bigg(\cfrac{t}{2}-2\bigg) = \cfrac{T}{2}$ $s$-vertices on $S_{n}^{s,t}$ and on $S_{n-2}^{s,t}$.
    Note that these $\cfrac{T}{2}$ $s$-vertices are not adjacent to other $t$-vertices from $W_{n-1}^{s,t}$.
    
    \item Assume $w$ is adjacent to the left to an $s$-vertex $v_1$ and to the right to a $t$-vertex $w_2$ on the same sphere $S_{n-1}^{s,t}$.
    Then $w$ is adjacent to $\bigg(\cfrac{t}{2}-1\bigg)$ $s$-vertices on $S_{n}^{s,t}$ and on $S_{n-2}^{s,t}$.
    Out of these $\bigg(\cfrac{t}{2}-1\bigg)$ $s$-vertices, $\bigg(\cfrac{t}{2}-3\bigg)$ $s$-vertices are adjacent only to $w$.
    So we assign them to $w$.
    There is an $s$-vertex $v_3$ on $S_{n}^{s,t}$ such that $v_3 \sim w,w_2$.
    There is an $s$-vertex $v_4$ on $S_{n-2}^{s,t}$ such that $v_4 \sim w,w_2$.
    We assign $v_3$ to $w$.
    We leave $v_4$ unassigned for now.
    Note that $w_2$ fits one of the cases below.
    So $\bigg(\cfrac{t}{2}-2\bigg) = \cfrac{T}{2}$ $s$-vertices on $S_{n-2}^{s,t}$ and on $S_{n}^{s,t}$ are assigned to $w$.
    
    \item Assume $w$ is adjacent to two $t$-vertices $w_1,w_2$ on the same sphere $S_{n-1}^{s,t}$ such that $w_1$ lies to the left of $w$, and $w_2$ lies to the right of $w$.
    Then $w$ is adjacent to $\cfrac{t}{2}$ $s$-vertices on $S_{n}^{s,t}$ and on $S_{n-2}^{s,t}$.
    Out of these $\cfrac{t}{2}$ $s$-vertices, $\bigg(\cfrac{t}{2}-3\bigg)$ $s$-vertices are adjacent only to $w$.
    So we assign them to $w$.
    There are two $s$-vertices $v_3$, $v_4$ on $S_{n}^{s,t}$ such that $v_3 \sim w,w_1$ and $v_4 \sim w,w_2$.
    There is an $s$-vertex $v_5$ on $S_{n-2}^{s,t}$ such that $v_5 \sim w_1,w,w_2$.
    We assign $v_4$ to $w$.
    We leave $v_5$ unassigned for now.
    Because $w_1$ fits this case or the case above, $v_3$ was already assigned to $w_1$.
    Note that $w_2$ fits this case or the case below.
    So $\bigg(\cfrac{t}{2}-2\bigg) = \cfrac{T}{2}$ $s$-vertices on $S_{n-2}^{s,t}$ and on $S_{n}^{s,t}$ are assigned to $w$.
    
    \item Assume $w$ is adjacent to the left to a $t$-vertex $w_1$ and to the right to an $s$-vertex $v_2$ on the same sphere $S_{n-1}^{s,t}$.
    Then $w$ is adjacent to $\bigg(\cfrac{t}{2}-1\bigg)$ $s$-vertices on $S_{n}^{s,t}$ and on $S_{n-2}^{s,t}$.
    Out of these $\bigg(\cfrac{t}{2}-1\bigg)$ $s$-vertices, $\bigg(\cfrac{t}{2}-3\bigg)$ $s$-vertices are adjacent only to $w$.
    So we assign them to $w$.
    There is an $s$-vertex $v_3$ on $S_{n}^{s,t}$ such that $v_3 \sim w,w_1$.
    There is an $s$-vertex $v_4$ on $S_{n-2}^{s,t}$ such that $v_4 \sim w,w_1$.
    We assign $v_4$ to $w$.
    Because $w_1$ fits the second or third case above, $v_3$ was already assigned to $w_1$.
    So $\bigg(\cfrac{t}{2}-2\bigg) = \cfrac{T}{2}$ $s$-vertices on $S_{n-2}^{s,t}$ and on $S_{n}^{s,t}$ are assigned to $w$.
\end{enumerate}

In conclusion to each $t$-vertex from $W_{n-1}^{s,t}$ correspond $\bigg(\cfrac{t}{2}-2\bigg) = \cfrac{T}{2}$ $s$-vertices on $S_{n}^{s,t}$ and on $S_{n-2}^{s,t}$.
This proves relation (\ref{eq:rec-spheres-rel2}).

Using (\ref{eq:rec-spheres-svw}), (\ref{eq:rec-spheres-rel1}) and (\ref{eq:rec-spheres-rel2}) we obtain the recurrence relation we want to prove for the sphere lengths.
Namely,

$$|S_n^{s,t}| = S \cdot |V_{n-1}^{s,t}| + T \cdot |W_{n-1}^{s,t}| - |S_{n-2}^{s,t}| =$$
$$= S \cdot |V_{n-1}^{s,t}| + T \cdot |W_{n-1}^{s,t}| - 2 \cdot |S_{n-2}^{s,t}| + S \cdot |V_{n-3}^{s,t}| + T \cdot |W_{n-3}^{s,t}| - |S_{n-4}^{s,t}|=$$
$$= S \cdot (|V_{n-1}^{s,t}| + |V_{n-3}^{s,t}|) + T \cdot (|W_{n-1}^{s,t}| + |W_{n-3}^{s,t}|) - 2 \cdot |S_{n-2}^{s,t}| - |S_{n-4}^{s,t}|=$$
$$= S \cdot \cfrac{T}{2} \cdot |W_{n-2}^{s,t}| + T \cdot (|W_{n-1}^{s,t}| + |W_{n-3}^{s,t}|) - 2 \cdot |S_{n-2}^{s,t}| - |S_{n-4}^{s,t}|=$$
$$= \cfrac{T}{2} \cdot (2 \cdot |W_{n-1}^{s,t}| + S \cdot |W_{n-2}^{s,t}| + 2 \cdot |W_{n-3}^{s,t}|) - 2 \cdot |S_{n-2}^{s,t}| - |S_{n-4}^{s,t}|=$$
$$= \cfrac{T}{2} \cdot [2 \cdot (|S_{n-1}^{s,t}| - |V_{n-1}^{s,t}|) + S \cdot (|S_{n-2}^{s,t}| - |V_{n-2}^{s,t}|) + 2 \cdot (|S_{n-3}^{s,t}| - |V_{n-3}^{s,t}|)] - $$
$$- 2 \cdot |S_{n-2}^{s,t}| - |S_{n-4}^{s,t}|=$$
$$= \cfrac{T}{2} \cdot (|S_{n-1}^{s,t}| + S \cdot |S_{n-2}^{s,t}| + |S_{n-3}^{s,t}|) - 2 \cdot |S_{n-2}^{s,t}| - |S_{n-4}^{s,t}| +$$
$$+ \cfrac{T}{2} \cdot (|S_{n-1}^{s,t}| - 2 \cdot |V_{n-1}^{s,t}| - S \cdot |V_{n-2}^{s,t}| + |S_{n-3}^{s,t}| - 2 \cdot |V_{n-3}^{s,t}|) =$$
$$= \cfrac{T}{2} \cdot |S_{n-1}^{s,t}| + \bigg(\cfrac{S \cdot T}{2}-2\bigg) \cdot |S_{n-2}^{s,t}| + \cfrac{T}{2} \cdot |S_{n-3}^{s,t}| - |S_{n-4}^{s,t}| +$$
$$+ \cfrac{T}{2} \cdot [S \cdot |V_{n-2}^{s,t}| + T \cdot |W_{n-2}^{s,t}| - |S_{n-3}^{s,t}| - S \cdot |V_{n-2}^{s,t}| + |S_{n-3}^{s,t}| - 2 \cdot (|V_{n-1}^{s,t}| + |V_{n-3}^{s,t}|)] =$$
$$= \cfrac{T}{2} \cdot |S_{n-1}^{s,t}| + \bigg(\cfrac{S \cdot T}{2}-2\bigg) \cdot |S_{n-2}^{s,t}| + \cfrac{T}{2} \cdot |S_{n-3}^{s,t}| - |S_{n-4}^{s,t}| +$$
$$+ \cfrac{T}{2} \cdot \bigg(T \cdot |W_{n-2}^{s,t}| - 2 \cdot \cfrac{T}{2} \cdot |W_{n-2}^{s,t}|\bigg) =$$
$$= \cfrac{T}{2} \cdot |S_{n-1}^{s,t}| + \bigg(\cfrac{S \cdot T}{2}-2\bigg) \cdot |S_{n-2}^{s,t}| + \cfrac{T}{2} \cdot |S_{n-3}^{s,t}| - |S_{n-4}^{s,t}|$$

Next, using (\ref{eq:rec-spheres-svw}), (\ref{eq:rec-spheres-rel1}) and (\ref{eq:rec-spheres-rel2}), we obtain the recurrence relation we have to show for the number of $s$-vertices on spheres.
Namely,

$$|V_n^{s,t}| = \cfrac{T}{2} \cdot |W_{n-1}^{s,t}| - |V_{n-2}^{s,t}| =$$
$$= \cfrac{T}{2} \cdot |W_{n-1}^{s,t}| - 2 \cdot |V_{n-2}^{s,t}| + |V_{n-2}^{s,t}| =$$
$$= \cfrac{T}{2} \cdot |W_{n-1}^{s,t}| - 2 \cdot |V_{n-2}^{s,t}| + \cfrac{T}{2} \cdot |W_{n-3}^{s,t}| - |V_{n-4}^{s,t}| =$$
$$= \cfrac{T}{2} \cdot (|W_{n-1}^{s,t}| + |W_{n-3}^{s,t}|) - 2 \cdot |V_{n-2}^{s,t}| - |V_{n-4}^{s,t}| =$$
$$= \cfrac{T}{2} \cdot (|S_{n-1}^{s,t}| - |V_{n-1}^{s,t}| + |S_{n-3}^{s,t}| - |V_{n-3}^{s,t}|) - 2 \cdot |V_{n-2}^{s,t}| - |V_{n-4}^{s,t}| =$$
$$= \cfrac{T}{2} \cdot [(S \cdot |V_{n-2}^{s,t}| + T \cdot |W_{n-2}^{s,t}| - |S_{n-3}^{s,t}|) - |V_{n-1}^{s,t}| + |S_{n-3}^{s,t}| - |V_{n-3}^{s,t}|] - 2 \cdot |V_{n-2}^{s,t}| - |V_{n-4}^{s,t}| =$$
$$= \cfrac{T}{2} \cdot [S \cdot |V_{n-2}^{s,t}| + 2 \cdot (|V_{n-1}^{s,t}| + |V_{n-3}^{s,t}|) - |V_{n-1}^{s,t}| - |V_{n-3}^{s,t}|] - 2 \cdot |V_{n-2}^{s,t}| - |V_{n-4}^{s,t}| =$$
$$= \cfrac{T}{2} \cdot (|V_{n-1}^{s,t}| + S \cdot |V_{n-2}^{s,t}| + |V_{n-3}^{s,t}|) - 2 \cdot |V_{n-2}^{s,t}| - |V_{n-4}^{s,t}| =$$
$$= \cfrac{T}{2} \cdot |V_{n-1}^{s,t}| + \bigg( \cfrac{S \cdot T}{2} - 2 \bigg) \cdot |V_{n-2}^{s,t}| + \cfrac{T}{2} \cdot |V_{n-3}^{s,t}| - |V_{n-4}^{s,t}|$$

Note that the lengths of spheres and the number of $s$-vertices on spheres follow this recurrence relation.
Then, based on (\ref{eq:rec-spheres-svw}), the number of $t$-vertices on spheres also follow the same recurrence relation.
Relation (\ref{eq:rec-spheres-ve}) implies that the number of $(s,t)$-edges and the number of $s$-vertices on spheres satisfy the same recurrence relation.
Because the lengths of spheres and the number of $(s,t)$-edges on spheres follow the same recurrence relation, (\ref{eq:rec-spheres-sef}) implies that the number of $(t,t)$-edges on spheres also fulfill this recurrence relation.

\end{proof}

For the rest of the section we consider the notations introduced in the theorem above.

The recurrence relation given in Theorem \ref{theorem:sphere-recurrence} holds for any sequence of consecutive spheres no matter what the center of spheres is (e.g. a vertex, an edge, a convex loop).
Still, the four initial terms of the sequence depend on this center (i.e. whether $v$ is an $s$-vertex, a $t$-vertex, an edge, etc).

Below we present the initial terms of the lengths of spheres and the number of simplices mentioned in Theorem \ref{theorem:sphere-recurrence}.
The four initial terms being given, we can also compute the terms with negative indices of these sequences.
We add an extra column for the index $-1$ such that
$$|S_3^{s,t}| = \cfrac{T}{2} \cdot |S_2^{s,t}| + \bigg(\cfrac{S \cdot T}{2}-2\bigg) \cdot |S_1^{s,t}| + \cfrac{T}{2} \cdot |S_0^{s,t}| - |S_{-1}^{s,t}|$$
Considering this value, it will be easier to compute the coefficients of the general term.

On $S_0^{s,t}$ there are $0$ vertices connected to edges.
We can therefore determine the number of vertices on the sphere $S_1^{s,t}$:
\begin{itemize}
    \item if the center is an $s$-vertex: $|S_1^{s,t}|=s$, $|V_1^{s,t}|=0$, $|W_1^{s,t}|=s$,
    \item if the center is a $t$-vertex: $|S_1^{s,t}|=t$, $|V_1^{s,t}|=\cfrac{t}{2}$, $|W_1^{s,t}|=\cfrac{t}{2}$.
\end{itemize}
Then we use (\ref{eq:rec-spheres-svw}), (\ref{eq:rec-spheres-rel1}) and (\ref{eq:rec-spheres-rel2}) to compute the lengths of the spheres $S_2^{s,t}$, $S_3^{s,t}$.
Besides we find the number of $s$-vertices and $t$-vertices belonging to these spheres.
To obtain the number of edges on the same spheres, we refer to (\ref{eq:rec-spheres-sef}) and (\ref{eq:rec-spheres-ve}).

\begin{table}[ht]
    \centering
    \begin{tabular}{|c||d||l|l|l|l|}
         \hline
         n & $-1$ & $0$ & $1$ & $2$ & $3$ \\
         \hline
         \hline
         $|S_n^{s,t}|$ & $-s$ & $0$ & $s$ & $s\cdot T$ & $s\cdot\bigg[2\bigg(\cfrac{T}{2}\bigg)^2 + \bigg(\cfrac{S \cdot T}{2}-2\bigg) + 1\bigg]$ \\
         \hline
         $|V_n^{s,t}|$ & $0$ & $0$ & $0$ & $s\cdot\cfrac{T}{2}$ & $s\cdot\bigg[\bigg(\cfrac{T}{2}\bigg)^2\bigg]$ \\
         \hline
         $|W_n^{s,t}|$ & $-s$ & $0$ & $s$ & $s\cdot\cfrac{T}{2}$ & $s\cdot\bigg[\bigg(\cfrac{T}{2}\bigg)^2 + \bigg(\cfrac{S \cdot T}{2}-2\bigg) + 1\bigg]$ \\
         \hline
         $|E_n^{s,t}|$ & $0$ & $0$ & $0$ & $s\cdot T$ & $s\cdot\bigg[2\bigg(\cfrac{T}{2}\bigg)^2\bigg]$ \\
         \hline
         $|F_n^{s,t}|$ & $-s$ & $0$ & $s$ & 0 & $s\cdot\bigg[\bigg(\cfrac{S \cdot T}{2}-2\bigg) + 1\bigg]$ \\
         \hline
    \end{tabular}
    \caption{The initial terms of the sequences representing the number of certain simplices on spheres centered at an $s$-vertex}
    \label{tab:initial-terms-s}
\end{table}

\begin{table}[ht]
    \centering
    \begin{tabular}{|c||e||l|l|l|l|}
         \hline
         n & $-1$ & $0$ & $1$ & $2$ & $3$ \\
         \hline
         \hline
         $|S_n^{s,t}|$ & $-\cfrac{t}{2} \cdot 2$ & $0$ & $\cfrac{t}{2} \cdot 2$ & $\cfrac{t}{2} \cdot (S+T)$ & $\cfrac{t}{2}\cdot\bigg[2\bigg(\cfrac{T}{2}\bigg)^2 + 3\bigg(\cfrac{S \cdot T}{2}-2\bigg) + 4\bigg]$ \\
         \hline
         $|V_n^{s,t}|$ & $-\cfrac{t}{2}$ & $0$ & $\cfrac{t}{2}$ & $\cfrac{t}{2} \cdot \cfrac{T}{2}$ & $\cfrac{t}{2}\cdot\bigg[\bigg(\cfrac{T}{2}\bigg)^2 + \bigg(\cfrac{S \cdot T}{2}-2\bigg) + 1\bigg]$ \\
         \hline
         $|W_n^{s,t}|$ & $-\cfrac{t}{2}$ & $0$ & $\cfrac{t}{2}$ & $\cfrac{t}{2} \cdot \bigg( S + \cfrac{T}{2}\bigg)$ & $\cfrac{t}{2}\cdot\bigg[\bigg(\cfrac{T}{2}\bigg)^2 + 2\bigg(\cfrac{S \cdot T}{2}-2\bigg) + 3\bigg]$ \\
         \hline
         $|E_n^{s,t}|$ & $-\cfrac{t}{2} \cdot 2$ & $0$ & $\cfrac{t}{2} \cdot 2$ & $\cfrac{t}{2} \cdot T$ & $\cfrac{t}{2}\cdot\bigg[2\bigg(\cfrac{T}{2}\bigg)^2 + 2\bigg(\cfrac{S \cdot T}{2}-2\bigg) + 2\bigg]$ \\
         \hline
         $|F_n^{s,t}|$ & $\hspace{0.25cm} 0$ & $0$ & $0$ & $\cfrac{t}{2} \cdot S$ & $\cfrac{t}{2}\cdot\bigg[\bigg(\cfrac{S \cdot T}{2}-2\bigg) + 2\bigg]$ \\
         \hline
    \end{tabular}
    \caption{The initial terms of the sequences representing the number of certain simplices on spheres centered at a $t$-vertex}
    \label{tab:initial-terms-t}
\end{table}

\newpage

\begin{theorem}\label{theorem:area-in-sphere}
Let $X$ be an $(s,t)$-uniform simplicial complex and let $\gamma$ be a loop in $X$.
Let $(D, f)$ be a minimal filling diagram for $\gamma$.
Let $v$ be an interior vertex of $D$.
Then the area of a sphere centered at $v$ is equal to \begin{equation}\label{eq:area-sum-spheres}
    A_n^{s,t} = 2 \bigg(\sum_{k=0}^{n}{|S_k^{s,t}|}\bigg) - |S_n^{s,t}|
\end{equation}
\end{theorem}

\begin{proof}
For the proof see \cite{L-t-unif}, Theorem $3.3$.
\end{proof}

\subsection{Order-4 homogeneous linear recurrences in (s,t)-uniform simplicial complexes.}\label{sec:recurrences-s-t-unif}

Using Theorem \ref{theorem:sphere-recurrence}, we investigate the following symmetric order-4 homogeneous linear recurrence relation
\begin{equation}\label{eq:rec-order-4}
    x_n = (P-R) x_{n-1} + (PR-2) x_{n-2} + (P-R) x_{n-3} - x_{n-4}
\end{equation}
with $P,R$ integers, $1 < R$, $2 < P$, $R < P$.

The recurrence relation above follows from the order-$2$ linear recurrence relation $x_n = P x_{n-1} - x_{n-2}$ which was used to obtain certain Lucas sequences in $t$-uniform simplicial complexes (see \cite{L-t-unif}).
Namely,
$$x_n = P x_{n-1} - x_{n-2} =$$
$$= (P-R) x_{n-1} + R x_{n-1} - x_{n-2} =$$
$$= (P-R) x_{n-1} + R (P x_{n-2} - x_{n-3}) - x_{n-2} =$$
$$= (P-R) x_{n-1} + (PR-1) x_{n-2} - R x_{n-3} =$$
$$= (P-R) x_{n-1} + (PR-2) x_{n-2} + x_{n-2} - R x_{n-3} =$$
$$= (P-R) x_{n-1} + (PR-2) x_{n-2} + (P x_{n-3} - x_{n-4}) - R x_{n-3} =$$
$$= (P-R) x_{n-1} + (PR-2) x_{n-2} + (P-R) x_{n-3} - x_{n-4}$$

The characteristic equation of the recurrence relation above is a degree $4$ reciprocal equation
$$x^4 - (P-R) x^3 - (PR-2) x^2 - (P-R) x + 1 = 0$$
After reordering the equation, we get
$$(x^4 - P x^3 + x^2) + (R x^3 - PR x^2 + R x) + (x^2 - P x + 1) = 0$$
$$x^2(x^2 - P x + 1) + R x(x^2 - P x + 1) + (x^2 - P x + 1) = 0$$
$$(x^2 - P x +1)(x^2 + R x + 1) = 0$$

The first polynomial of the equation above has the discriminant $D = P^2 - 4$ and the roots
\begin{equation}\label{eq:req-p-roots}
    a = \cfrac{P + \sqrt{P^2 - 4}}{2},\quad
    b = \cfrac{P - \sqrt{P^2 - 4}}{2}
\end{equation}
Thus $a \cdot b = 1$.
Also $a + b = P$, $a - b = \sqrt{P^2 - 4}$ and $a^2 - b^2 = P\sqrt{P^2 - 4}$.
As $2 < P$, we have $0 < b < 1 < a < P$.

The second polynomial of the equation above has the discriminant $D = R^2 - 4$ and the roots
\begin{equation}\label{eq:req-r-roots}
    c = \cfrac{-R + \sqrt{R^2 - 4}}{2},\quad
    d = \cfrac{-R - \sqrt{R^2 - 4}}{2}
\end{equation}
Thus $c \cdot d = 1$.
Also $c + d = -R$, $c - d = \sqrt{R^2 - 4}$ and $c^2 - d^2 = -R\sqrt{R^2 - 4}$.

There are $3$ cases to be analyzed.
We present them below.
\begin{enumerate}[label=(\alph*), ref=\alph*]
    \item \label{case-a} If $2 < R < P$ then $D > 0$ and there are distinct real roots such that $-P < -R < d < -1 < c < 0$.
    So
    \begin{equation}\label{eq:relation-solutions}
        0 < b < |c| < 1 < |d| < a
    \end{equation}

    \item \label{case-b} If $R = 2$ then $D = 0$ and $c = d = -1$.
    Hence
    \begin{equation}\label{eq:relation-solutions-duplicate}
        0 < b < |c| = 1 = |d| < a
    \end{equation}

    \item \label{case-c} If $1 < R < 2 < P$ then $D < 0$ and there are distinct complex roots.
\end{enumerate}

In the cases (\ref{case-a}) and (\ref{case-c}), the roots of the characteristic equation are distinct ($r_1 = a$, $r_2 = b$, $r_3 = c$, $r_4 = d$).
Then, due to (\ref{eq:rec-rel-general-term}) and using the notation $A = k_1$, $B = k_2$, $C = k_3$, $D = k_4$, each solution of the recurrence relation has the form
\begin{equation}\label{eq:general-term-distinct}
    x_n = \sum_{i=1}^{4} k_i r_i^n = A a^n + B b^n + C c^n + D d^n
\end{equation}

In the case (\ref{case-b}), two of the roots of the characteristic equation are duplicate ($r_1 = a$, $r_2 = b$, $r_3 = c = d = -1$).
Then, due to (\ref{eq:rec-rel-general-term-duplicate}) and using the notation $A = k_{11}$, $B = k_{21}$, $C = k_{31}$, $D = k_{32}$, each solution of the recurrence relation has the form
\begin{equation}\label{eq:general-term-duplicate}
    x_n = \sum_{i=1}^{3} \sum_{j=1}^{p_i} k_{ij} \cdot n^{j-1} \cdot r_i^n
    = A a^n + B b^n + (C + nD) c^n
\end{equation}

In the Tables \ref{tab:initial-terms-s} and \ref{tab:initial-terms-t} the initial terms of the sequences of simplices on spheres are given.
We are only interested in the sequences with these initial terms.
We consider the term $x_{-1}$ such that
$$x_3 = (P-R) x_2 + (PR-2) x_1 + (P-R) x_0 - x_{-1}$$
This will simplify the system of equations we have to solve.
We note that the initial terms of the sequence $(x_{n})_{n \geq -1}$ can take only certain values:
\begin{center}
    \begin{tabular}{|r|r|r|r|}
        \hline
        $x_{-1}$ & $x_0$ & $x_1$ & $x_2$ \\
        \hline
        $-v_1$ & $0$ & $v_1$ & $v_2$ \\
        \hline
    \end{tabular}
\end{center}
where either $v_1 > 0$, $v_2 \ge 0$ or $v_1 = 0$, $v_2 > 0$.

Next we show that for these initial terms we have $A \neq 0$.

In the cases (\ref{case-a}) and (\ref{case-c}), we apply (\ref{eq:general-term-distinct}) for $-1 \leq n \leq 2$.
We obtain the following system of equations with the unknowns $A, B, C, D$:
\begin{equation*}\label{eq:system-distinct-a-not-0-1}
\systeme[ABCD]{
    a^{-1} A + b^{-1} B + c^{-1} C + d^{-1} D = -v_1,
    a^0    A + b^0    B + c^0    C + d^0    D = 0,
    a^1    A + b^1    B + c^1    C + d^1    D = v_1,
    a^2    A + b^2    B + c^2    C + d^2    D = v_2
}
\end{equation*}
Since $a \cdot b = 1$ and $c \cdot d = 1$, the above system of equations becomes
\begin{equation*}\label{eq:system-distinct-a-not-0-2}
\sysautonum{(A_{*})}
\systeme[ABCD]{
    b A + a B + d C + c D = -v_1,
    A + B + C + D = 0,
    a A + b B + c C + d D = v_1,
    a^2 A + b^2 B + c^2 C + d^2 D = v_2
}
\end{equation*}
From $(A_1) + (A_3)$, we get
$(a + b)(A + B) + (c + d)(C + D) = 0$.
By $(A_2)$, we have $C + D = -(A + B)$.
This implies that $[(a + b) - (c + d)](A + B) = 0$.
However $a + b = P > 0$ and $c + d = -R < 0$.
This yields that $(a + b) - (c + d) > 0$.
Hence $A + B = 0$ and $C + D = 0$.

According to $(A_3)$, we have
$$a A + (-b A + b A) + b B + c C + (- d C + d C) + d D = v_1$$
$$(a-b) A + b (A + B) + (c-d) C + d (C + D) = v_1$$
\begin{equation} \label{eq:system-duplicate-v1-v2-1}
(a-b) A + (c-d) C = v_1
\end{equation}
Multiplying by $c+d = -R$, we get
$$(a-b)(c+d) A + (c^2-d^2) C = (c+d) \cdot v_1$$
Expressing $a,b,c$ and $d$ in terms of $P$ and $R$, we have
\begin{equation} \label{eq:equation-v1-v2-1}
    -\big(R\sqrt{P^2-4}\big) A - \big(R\sqrt{R^2 - 4}\big) C = -R \cdot v_1
\end{equation}
We argue similarly for $(A_4)$.
It follows that
$$a^2 A + (-b^2 A + b^2 A) + b^2 B + c^2 C + (- d^2 C + d^2 C) + d^2 D = v_2$$
$$(a^2-b^2) A + b^2 (A + B) + (c^2-d^2) C + d^2 (C + D) = v_2$$
$$(a^2-b^2) A + (c^2-d^2) C = v_2$$
Expressing $a,b,c$ and $d$ in terms of $P$ and $R$, we have
\begin{equation} \label{eq:equation-v1-v2-2}
    \big(P\sqrt{P^2 - 4}\big) A - \big(R\sqrt{R^2 - 4}\big) C = v_2
\end{equation}
Subtracting (\ref{eq:equation-v1-v2-1}) from (\ref{eq:equation-v1-v2-2}), we get
$$\big[(P+R)\sqrt{P^2 - 4}\big] A = v_2 + R \cdot v_1$$
So
\begin{equation}\label{eq:distinct-sphere-v1-v2-total-A}
    A = \cfrac{v_2 + R \cdot v_1}{(P+R)\sqrt{P^2 - 4}}
\end{equation}
We know that \begin{center}$B = -A = -\cfrac{v_2 + R \cdot v_1}{(P+R)\sqrt{P^2 - 4}}$ \end{center}
Computing similarly, we obtain $C$ and $D$.
Multiplying the result in (\ref{eq:system-duplicate-v1-v2-1}) by $a+b = P$, it follows that
$$(a^2-b^2) A + (a+b)(c-d) C = (a+b) \cdot v_1$$
Expressing $a, b, c$ and $d$ in terms of $P$ and $R$, we get
\begin{equation} \label{eq:equation-v1-v2-3}
    (P\sqrt{P^2-4}) A + (P\sqrt{R^2 - 4}) C = P \cdot v_1
\end{equation}
Subtracting (\ref{eq:equation-v1-v2-2}) from (\ref{eq:equation-v1-v2-3}), we obtain
$$\big[(P+R)\sqrt{R^2 - 4}\big] C = P \cdot v_1 - v_2$$
So
$$C = \cfrac{P \cdot v_1 - v_2}{(P+R)\sqrt{R^2 - 4}}$$
We know that $D = -C = -\cfrac{P \cdot v_1 - v_2}{(P+R)\sqrt{R^2 - 4}}$.

\begin{table}[ht]
    \centering
    \begin{tabular}{|c|c|c|c|}
        \hline
        $A$ & $B$ & $C$ & $D$ \\
        \hline
        $\cfrac{v_2 + R \cdot v_1}{(P+R)\sqrt{P^2 - 4}}$
        & $-\cfrac{v_2 + R \cdot v_1}{(P+R)\sqrt{P^2 - 4}}$
        & $\cfrac{P \cdot v_1 - v_2}{(P+R)\sqrt{R^2 - 4}}$
        & $-\cfrac{P \cdot v_1 - v_2}{(P+R)\sqrt{R^2 - 4}}$ \\
        \hline
    \end{tabular}
    \caption{The coefficients of the general term of the sequence $(x_{n})_{n \geq 0}$ in the cases (\ref{case-a}) and (\ref{case-c})}
    \label{tab:coef-x-a-c-v1-v2}
\end{table}

Recall the initial terms we considered in the cases (\ref{case-a}) and (\ref{case-c}).
Then, given that $R > 1$ and that either $v_1 > 0$, $v_2 \ge 0$ or $v_1 = 0$, $v_2 > 0$, we have $v_2 + R \cdot v_1 > 0$.
So $A > 0$.

In the case (\ref{case-b}) we apply (\ref{eq:general-term-duplicate}) for $-1 \leq n \leq 2$.
We obtain the following system of equations with the unknowns $A, B, C, D$:
\begin{equation*}\label{eq:system-duplicate-a-not-0-1}
\systeme[ABCD]{
    a^{-1} A + b^{-1} B + c^{-1} C - c^{-1} D = -v_1,
    a^0    A + b^0    B + c^0    C = 0,
    a^1    A + b^1    B + c^1    C + c^1 D = v_1,
    a^2    A + b^2    B + c^2    C + 2 c^2 D = v_2
}
\end{equation*}
We know that $a \cdot b = 1$ and $c = -1$.
Then the above system of equations becomes
\begin{equation*}\label{eq:system-duplicate-a-not-0-2}
\sysautonum{(B_{*})}
\systeme[ABCD]{
    b A + a B - C + D = -v_1,
    A + B + C = 0,
    a A + b B - C - D = v_1,
    a^2 A + b^2 B + C + 2D = v_2
}
\end{equation*}
From $(B_1) + (B_3)$, we get
$(a + b)(A + B) - 2C = 0$.
By $(B_2)$, we have $C = -(A + B)$.
So this implies
$[(a + b) + 2](A + B) = 0$.
However, since $a + b = P > 0$, we have $(a + b) + 2 > 0$.
Hence $A + B = 0$ and $C = 0$.
Then the system of equations becomes
\begin{equation*}\label{eq:system-duplicate-v1-v2-3}
\sysautonum{(C_{*})}
\systeme[ABCD]{
    b A + a B + D = -v_1,
    A + B = 0,
    a A + b B - D = v_1,
    a^2 A + b^2 B + 2D = v_2
}
\end{equation*}
From $(C_3) -b \cdot (C_2)$, we get $(a-b) A - D = v_1$.
From $(C_4) -b^2 \cdot (C_2)$, we have $(a^2-b^2) A + 2D = v_2$.
Expressing $a$ and $b$ in terms of $P$, we obtain
\begin{equation*}\label{eq:system-duplicate-v1-v2-4}
\sysautonum{(D_{*})}
\systeme[ABCD]{
    \sqrt{P^2-4} A - D = v_1,
    (P\sqrt{P^2-4}) A + 2D = v_2
}
\end{equation*}
From $(D_2) + 2 (D_1)$, we get $[(P+2)\sqrt{P^2-4}] A = v_2 + 2 \cdot v_1$.
So
\begin{equation}\label{eq:duplicate-sphere-v1-v2-total-A}
    A = \cfrac{v_2 + 2 \cdot v_1}{(P+2)\sqrt{P^2-4}}
\end{equation}
We know that $B = -A = -\cfrac{v_2 + 2 \cdot v_1}{(P+2)\sqrt{P^2-4}}$ and that $C = 0$.
Besides $(D_1)$ implies that
$$D = \cfrac{v_2 + 2 \cdot v_1}{P+2} - v_1
= \cfrac{(v_2 + 2 \cdot v_1) - (P+2) \cdot v_1}{P+2}
= \cfrac{v_2 - P \cdot v_1}{P+2}
= -\cfrac{P \cdot v_1 - v_2}{P+2}$$

\begin{table}[ht]
    \centering
    \begin{tabular}{|c|c|c|c|}
        \hline
        $A$ & $B$ & $C$ & $D$ \\
        \hline
        $\cfrac{v_2 + 2 \cdot v_1}{(P+2)\sqrt{P^2-4}}$
        & $-\cfrac{v_2 + 2 \cdot v_1}{(P+2)\sqrt{P^2-4}}$
        & $0$
        & $-\cfrac{P \cdot v_1 - v_2}{P+2}$ \\
        \hline
    \end{tabular}
    \caption{The coefficients of the general term of the sequence $(x_{n})_{n \geq 0}$ in the case (\ref{case-b})}
    \label{tab:coef-x-b-v1-v2}
\end{table}

Recall the initial terms we considered in the case (\ref{case-b}).
Then given that either $v_1 > 0$, $v_2 \ge 0$ or $v_1 = 0$, $v_2 > 0$, we have $v_2 + 2 \cdot v_1 > 0$.
So $A > 0$.

In conclusion in all cases we have $A \neq 0$.

Further we compute the limit of the ratio between two consecutive terms of the sequence $(x_n)_{n \ge 0}$.

In the case (\ref{case-a}), relation (\ref{eq:relation-solutions}) implies that
\begin{equation}\label{eq:limit-sol-ratio-distinct-1}
    \lim\limits_{n\to\infty} \bigg(\cfrac{b}{a}\bigg)^n =
    \lim\limits_{n\to\infty} \bigg(\cfrac{c}{a}\bigg)^n =
    \lim\limits_{n\to\infty} \bigg(\cfrac{d}{a}\bigg)^n = 0
\end{equation}

In the case (\ref{case-b}), relation (\ref{eq:relation-solutions-duplicate}) implies that
\begin{equation}\label{eq:limit-sol-ratio-duplicate-2}
    \lim\limits_{n\to\infty} \bigg(\cfrac{b}{a}\bigg)^n =
    \lim\limits_{n\to\infty} \bigg(\cfrac{c}{a}\bigg)^n = 0
\end{equation}

In the case (\ref{case-c}), we note that $c$ and $d$ are conjugate complex numbers.
We have $c \cdot d = 1$.
Then, because $c \cdot d = |c|^2 = |d|^2 = 1$, it follows that $|c| = |d| = 1$.
Since $a$ is a real number and $a > 1$, we get $\bigg|\cfrac{c}{a}\bigg| < 1$ and $\bigg|\cfrac{d}{a}\bigg| < 1$.
So
\begin{equation}\label{eq:limit-sol-ratio-distinct-3}
    \lim\limits_{n\to\infty} \bigg(\cfrac{b}{a}\bigg)^n =
    \lim\limits_{n\to\infty} \bigg(\cfrac{c}{a}\bigg)^n =
    \lim\limits_{n\to\infty} \bigg(\cfrac{d}{a}\bigg)^n = 0
\end{equation}

Hence, in the cases (\ref{case-a}) and (\ref{case-c}), due to (\ref{eq:general-term-distinct}), (\ref{eq:limit-sol-ratio-distinct-1}) and (\ref{eq:limit-sol-ratio-distinct-3}), we have
$$\lim\limits_{n\to\infty} \cfrac{x_n}{x_{n-1}} =
\lim\limits_{n\to\infty} \cfrac{A a^n + B b^n + C c^n + D d^n}{A a^{n-1} + B b^{n-1} + C c^{n-1} + D d^{n-1}} =$$
$$= \lim\limits_{n\to\infty}\cfrac
{a^n \bigg( A \cfrac{a^n}{a^n} + B \cfrac{b^n}{a^n} + C \cfrac{c^n}{a^n} + D \cfrac{d^n}{a^n} \bigg)}
{a^{n-1} \bigg( A \cfrac{a^{n-1}}{a^{n-1}} + B \cfrac{b^{n-1}}{a^{n-1}} + C \cfrac{c^{n-1}}{a^{n-1}} + D \cfrac{d^{n-1}}{a^{n-1}} \bigg)} =$$
$$= \lim\limits_{n\to\infty} a \cdot \cfrac
{A + B \bigg(\cfrac{b}{a}\bigg)^n + C \bigg(\cfrac{c}{a}\bigg)^n + D \bigg(\cfrac{d}{a}\bigg)^n}
{A + B \bigg(\cfrac{b}{a}\bigg)^{n-1} + C \bigg(\cfrac{c}{a}\bigg)^{n-1} + D \bigg(\cfrac{d}{a}\bigg)^{n-1}} = a$$

In the case (\ref{case-b}), due to (\ref{eq:general-term-duplicate}) and (\ref{eq:limit-sol-ratio-duplicate-2}), we get
$$\lim\limits_{n\to\infty} \cfrac{x_n}{x_{n-1}} =
\lim\limits_{n\to\infty} \cfrac{A a^n + B b^n + \big(C + nD\big) c^n}{A a^{n-1} + B b^{n-1} + \big(C + (n-1)D\big) c^{n-1}} =$$
$$= \lim\limits_{n\to\infty}\cfrac
{a^n \bigg[ A \cfrac{a^n}{a^n} + B \cfrac{b^n}{a^n} + \big(C + nD\big) \cfrac{c^n}{a^n} \bigg]}
{a^{n-1} \bigg[ A \cfrac{a^{n-1}}{a^{n-1}} + B \cfrac{b^{n-1}}{a^{n-1}} + \big(C + (n-1)D\big) \cfrac{c^{n-1}}{a^{n-1}} \bigg]} =$$
$$= \lim\limits_{n\to\infty} a \cdot \cfrac
{A + B \bigg(\cfrac{b}{a}\bigg)^n + \big(C + nD\big) \bigg(\cfrac{c}{a}\bigg)^n}
{A + B \bigg(\cfrac{b}{a}\bigg)^{n-1} + \big(C + (n-1)D\big) \bigg(\cfrac{c}{a}\bigg)^{n-1}} = a$$

So, because $A \neq 0$, in all cases we have
\begin{equation}\label{eq:limit-ratio-succ-terms}
    \lim\limits_{n\to\infty} \cfrac{x_n}{x_{n-1}} = a = \cfrac{1}{b} > 1
\end{equation}

Next, using Stolz-Ces\`aro theorem, we compute, for the sequence $(x_n)_{n \ge 0}$, the limit of the ratio of the sum of its first $n+1$ terms over its $(n+1)$th term.
Let the initial terms $x_{i}, 0 \leq i \leq 3$ of the sequence be non-negative and in strictly increasing order.
We assume, by induction, that $x_{k} > x_{k-1}$, $k < n$.
We show that $x_n > x_{n-1}$.
We know that $1 < R$, $2 < P$, $R < P$.
Recall $P$ and $R$ are integer values.
So $P-R > 1$ and $PR-2 > 1$.
It follows that
$$x_n = (P-R) x_{n-1} + (PR-2) x_{n-2} + (P-R) x_{n-3} - x_{n-4} >$$
$$> x_{n-1} + x_{n-2} + x_{n-3} - x_{n-4} > x_{n-1} + x_{n-2} > x_{n-1}$$
So the sequence $(x_n)_{n \ge 0}$ is strictly increasing.

From $x_n > x_{n-1} + x_{n-2}$ and $x_{n-1} > x_{n-2}$, we get $x_n > 2 \cdot x_{n-2}$.
So $x_n > 2^{\frac{n}{2}} \cdot x_2$ for $n$ even, and $x_n > 2^{\frac{n-1}{2}} \cdot x_1$ for $n$ odd.
Because $x_1$ and $x_2$ are non-negative, it follows that $\lim \limits_{n\to\infty} x_n = \infty$.
So the sequence $(x_n)_{n \ge 0}$ is divergent.

Stolz-Ces\`aro theorem states that given two sequences of real numbers $(a_n)_{n \ge 0}$ and $(b_n)_{n \ge 0}$, with $(b_n)_{n \ge 0}$ being a strictly monotone and divergent sequence, if the following limit exists $\lim \limits_{n\to\infty} \cfrac{a_{n+1} - a_n}{b_{n+1} - b_n} = l$, then $\lim \limits_{n\to\infty} \cfrac{a_{n}}{b_n} = l$.

We set $(a_n)_{n \ge 0} = \bigg(\sum\limits_{k = 0}^n x_k\bigg)_{n \ge 0}$ and $(b_n)_{n \ge 0} = (x_n)_{n \ge 0}$.
It is proven above that $(x_n)_{n \ge 0}$ is a strictly increasing, divergent sequence.
Then, using (\ref{eq:limit-ratio-succ-terms}), we have
$$\lim \limits_{n\to\infty} \cfrac{a_{n+1} - a_n}{b_{n+1} - b_n} =
\lim \limits_{n\to\infty} \cfrac{\sum\limits_{k = 0}^{n+1} x_k - \sum\limits_{k = 0}^n x_k}{x_{n+1} - x_n} =$$
$$= \lim \limits_{n\to\infty} \cfrac{x_{n+1}}{x_{n+1} - x_n}
= \lim \limits_{n\to\infty} \cfrac{1}{1 - \cfrac{x_n}{x_{n+1}}}
= \cfrac{1}{1 - \cfrac{1}{a}}
= \cfrac{a}{a - 1}$$
Because the above limit exists, the Stolz-Ces\`aro theorem implies that
\begin{equation}\label{eq:limit-sum-over-last-a-b}
    \lim\limits_{n\to\infty} \cfrac{\sum\limits_{k = 0}^n x_k}{x_n} = \cfrac{a}{a-1}
\end{equation}
Using (\ref{eq:req-p-roots}) and (\ref{eq:limit-sum-over-last-a-b}), we can express the limit in terms of $P$ (as computed in \cite{L-t-unif})
\begin{equation}\label{eq:limit-sum-over-last-p}
    \lim\limits_{n\to\infty} \cfrac{\sum\limits_{k = 0}^n x_k}{x_n} =
    \cfrac{1 + \sqrt{\cfrac{P+2}{P-2}}}{2}
\end{equation}

We note that the limit depends only on the roots $a$ and $b$ (not on the roots $c$ and $d$).
Therefore, because $a$ and $b$ depend only on $P$, the limit also depends only on $P$ (not on $R$).

In the next theorem we find, for spheres in minimal filling diagrams associated to loops in $(s,t)$-uniform simplicial complexes, the limit of the sequence $\bigg(\cfrac{A_{n}^{s,t}}{|S_{n}^{s,t}|}\bigg)_{n \geq 0}$ as $n \rightarrow \infty$.
It is important to note that for such complexes this sequence is not strictly increasing as it was for $t$-uniform simplicial complexes (see \cite{L-t-unif}).
So for $(s,t)$-uniform complexes we can no longer find an inequality between the area and the lengths of spheres using as constant the limit of their ratio.

\begin{theorem}\label{theorem:area-length-ratio-s-t-unif}
In $(s,t)$-uniform simplicial complexes, $s, t \ge 6$, $t$ even, except for $s = t = 6$, the following holds:
\begin{equation}\label{eq:limit-area-length-s-t-uniform}
    \lim_{n\to\infty} \cfrac{A_n^{s,t}}{|S_n^{s,t}|} = \sqrt{\cfrac{T + 8 + \sqrt{T^2 + 8ST}}{T - 8 + \sqrt{T^2 + 8ST}}}
\end{equation}
where $T = t-4$ and $S = s-4$.
The limit can be also expressed as
\begin{equation}\label{eq:limit-area-length-s-t-uniform-p}
    \lim_{n\to\infty} \cfrac{A_n^{s,t}}{|S_n^{s,t}|} = \sqrt{\cfrac{P + 2}{P - 2}}
\end{equation}
where $P = \cfrac{T + \sqrt{T^2 + 8ST}}{4}$.
\end{theorem}

\begin{proof}

If $v$ is an $s$-vertex, the lengths of spheres $|S_n^{s,t}|$ are the terms of a sequence $(x_n)_{n \ge 0}$ multiplied by $s$.
The initial terms of this sequence are given in Table \ref{tab:initial-terms-s}.

If $v$ is a $t$-vertex, the lengths of spheres $|S_n^{s,t}|$ are the terms of a sequence $(x_n)_{n \ge 0}$ multiplied by $\cfrac{t}{2}$.
The initial terms of this sequence are given in Table \ref{tab:initial-terms-t}.

Based on (\ref{eq:rec-spheres}) (Theorem \ref{theorem:sphere-recurrence}), the lengths of spheres centered at a vertex are sequences that fulfill a certain recurrence relation with parameters expressed in terms of $S$ and $T$.
Identifying the coefficients of the recurrence relations (\ref{eq:rec-spheres}) and (\ref{eq:rec-order-4}), we obtain the following system of equations:
\begin{equation} \label{eq:system-hyperbolas}
    \begin{cases}
        P - R = \cfrac{T}{2} \\
        P \cdot R = \cfrac{S \cdot T}{2}
    \end{cases}
\end{equation}
This implies that
$P \cdot \bigg(P - \cfrac{T}{2}\bigg) = \cfrac{S \cdot T}{2}$
and hence \begin{center}
$2P^2 - TP - ST = 0$ \end{center}
This equation has the roots
\begin{center}
$P_{1,2} = \cfrac{T \pm \sqrt{T^2 + 8ST}}{4}$
\end{center}
Because $P > 2$, we consider
\begin{equation}\label{eq:p-in-terms-of-s-t}
    P = \cfrac{T + \sqrt{T^2 + 8ST}}{4}
\end{equation}

By (\ref{eq:limit-sum-over-last-p}), if $v$ is an $s$-vertex, we have
\begin{equation}\label{eq:spaheres-ratio-sum-last-s}
    \lim_{n\to\infty} \cfrac{\sum_{k=0}^n{|S_k^{s,t}|}}{|S_n^{s,t}|} =
    \lim_{n\to\infty} \cfrac{\sum_{k=0}^n{\big( s \cdot x_k \big) }}{s \cdot x_n} =
    \lim_{n\to\infty} \cfrac{\sum_{k=0}^n{x_k}}{x_n} = \cfrac{1 + \sqrt{\cfrac{P+2}{P-2}}}{2}
\end{equation}

Similarly, if $v$ is a $t$-vertex, we get 
\begin{equation}\label{eq:spaheres-ratio-sum-last-t}
    \lim_{n\to\infty} \cfrac{\sum_{k=0}^n{|S_k^{s,t}|}}{|S_n^{s,t}|} =
    \lim_{n\to\infty} \cfrac{\sum_{k=0}^n{x_k}}{x_n} = \cfrac{1 + \sqrt{\cfrac{P+2}{P-2}}}{2}
\end{equation}

Using (\ref{eq:area-sum-spheres}), (\ref{eq:p-in-terms-of-s-t}), (\ref{eq:spaheres-ratio-sum-last-s}) and (\ref{eq:spaheres-ratio-sum-last-t}), for a sphere we get the limit, when $n$ goes to infinity, of the ratio of its area over its length.
Namely,
$$\lim_{n\to\infty} \cfrac{A_n^{s,t}}{|S_n^{s,t}|} =
\lim_{n\to\infty} \cfrac{2 \big(\sum_{k=0}^{n}{|S_k^{s,t}|}\big) - |S_n^{s,t}|}{|S_n^{s,t}|} =
\lim_{n\to\infty} 2 \cdot \cfrac{\sum_{k=0}^{n}{|S_k^{s,t}|}}{|S_n^{s,t}|} - 1 =$$
$$= 2 \cdot \cfrac{1 + \sqrt{\cfrac{P+2}{P-2}}}{2} - 1
=\sqrt{\cfrac{P+2}{P-2}} =$$
$$=\sqrt{\cfrac{T + 8 + \sqrt{T^2 + 8ST}}{T - 8 + \sqrt{T^2 + 8ST}}}$$
\end{proof}

\begin{remark}
For a $t$-uniform simplicial complex, we have $s = t$.
Since $S = T = t-4$, and due to (\ref{eq:p-in-terms-of-s-t}), we obtain
$$P = \cfrac{T + \sqrt{T^2 + 8ST}}{4} = \cfrac{T + T\sqrt{1 + 8}}{4} =  T = t-4$$
\end{remark}
Note that in this case the limit of the ratio of the area over the length of a sphere is the same as the one computed in \cite{L-t-unif}.
Namely,
$$\lim_{n\to\infty} \cfrac{A_n^{s,t}}{|S_n^{s,t}|} =
\sqrt{\cfrac{P + 2}{P - 2}} =
\sqrt{\cfrac{t - 2}{t - 6}}$$

In Figure \ref{fig:table-hyperbolas-s-t} we represent some $(s,t)$-uniform simplicial complexes.
The graph has $s$ on the horizontal axis and $t$ on the vertical axis.
The points at the grid intersections represent the $(s,t)$-uniform simplicial complexes.
In this graph we are only interested in some $(s,t)$-uniform simplicial complexes. Namely,
\begin{itemize}
    \item for $s \ge 6$, $t \ge 6$, $t$ even, in the $(s,t)$-uniform complexes colored at the red dots and which lie at the intersection points of the dark colored axes;
    \item  for $s = t$, $t$ odd, in the $(s,t)$-uniform complexes colored at the green dots and which lie on the main diagonal.
\end{itemize}

\begin{figure}[ht]
    \centering
    \includegraphics[width=\textwidth]{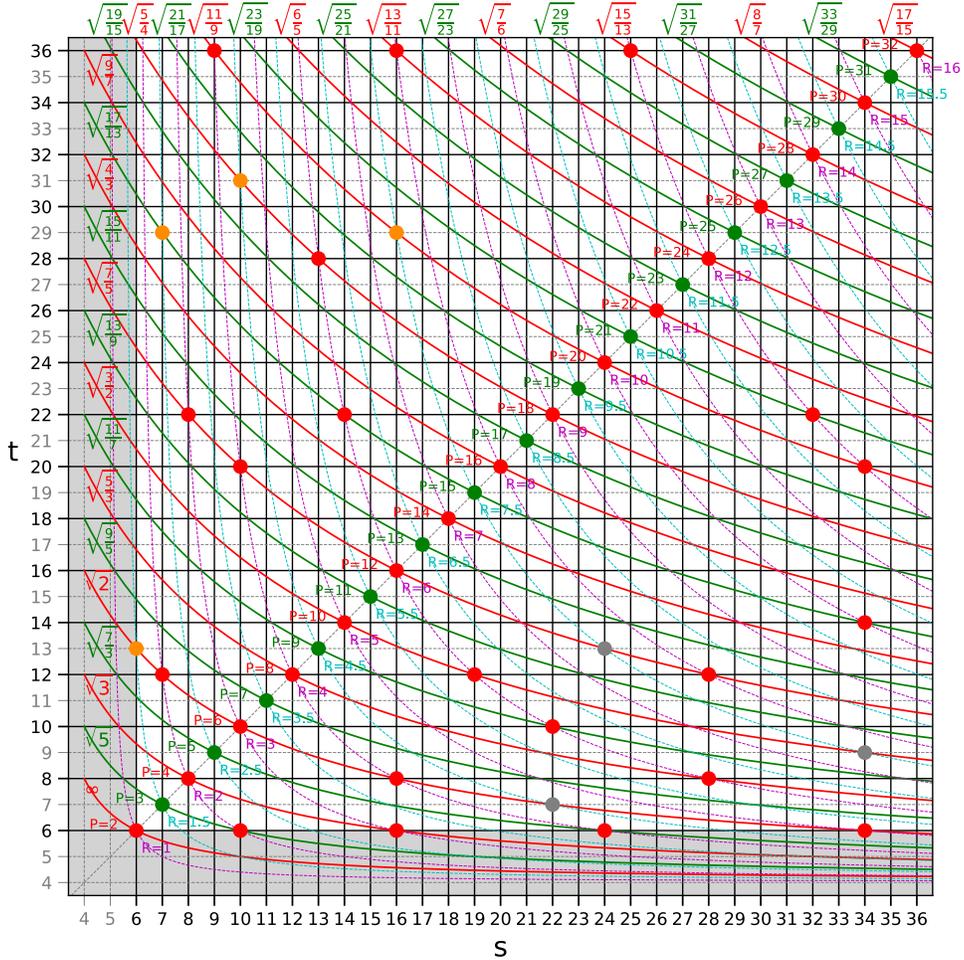}
    \caption{Graph of hyperbolas for (s,t)-uniform simplicial complexes}
    \label{fig:table-hyperbolas-s-t}
\end{figure}

For fixed values of $P$, the system of equations given in (\ref{eq:system-hyperbolas}) becomes the equation of a hyperbola in terms of $S$ and $T$:
$$T \cdot (S+P) = 2 \cdot P^2$$
For clarity, we represent only those hyperbolas that pass through points corresponding to $t$-uniform complexes.
We represent all $(s,t)$-uniform complexes that lie on these hyperbolas.
Then, given (\ref{eq:limit-area-length-s-t-uniform-p}), for all $(s,t)$-uniform complexes lying on one such hyperbola, the limit of the ratio of sphere area over sphere length is the same. This is the case because this limit is expressed only in terms of $P$.
The red hyperbolas correspond to even values of $P$.
Each such hyperbola contains a point that corresponds to a $t$-uniform complex for $t$ even.
The green hyperbolas correspond to odd values of $P$.
Each such hyperbola contains a point that corresponds to a $t$-uniform complex for $t$ odd.
On the left side and on the top side of the graph we represent the ratio of sphere area over sphere length for the complexes that lie on these hyperbolas.

Similarly, for fixed values of $R$, the system of equations given in (\ref{eq:system-hyperbolas}) becomes the equation of a hyperbola.
However, for the $(s,t)$-uniform complexes lying on one such hyperbola, the limit of the ratio of sphere area over sphere length is no longer the same.
The magenta hyperbolas correspond to integer values of $R$.
Each such hyperbola contains a point that corresponds to a $t$-uniform complex for $t$ even.
If we ignore the restriction that $R$ takes only integer positive values, then the cyan hyperbolas correspond to fractional values of $R$.
Each such hyperbola contains a point that corresponds to a $t$-uniform complex for $t$ odd.

The orange and gray dots correspond to $(s,t)$-uniform complexes for $t$ odd.
In such complexes, the $t$-vertices must have more $t$-vertex neighbours than $s$-vertex neighbours.
So for the orange dots, which have $t > s$, we assume that $\lim _{n \rightarrow \infty} \cfrac{A_{n}^{s,t}}{|S_{n}^{s,t}|}$ is less or equal to the value associated to the hyperbola containing the pair $(s,t)$.

One can identify on the graph a few sets of $(s,t)$-uniform simplicial complexes with the same limit of the ratio of sphere area over sphere length.
Namely,
\renewcommand{\arraystretch}{2}
\begin{table}[ht]
    \begin{tabular}{|l|c|}
        \hline
        $(s,t)$-uniform complexes & $\lim\limits_{n\to\infty} \cfrac{A_n^{s,t}}{|S_n^{s,t}|}$ \\
        \hline
        $(10, 6), (7, 7)$ & $\sqrt{5}$ \\
        \hline
        $(16, 6), (8, 8)$ & $\sqrt{3}$ \\
        \hline
        $(24, 6), (9, 9)$ & $\sqrt{\cfrac{7}{3}}$ \\
        \hline
        $(34, 6), (16, 8)$, $(10, 10), (7, 12)$ & $\sqrt{2}$ \\
        \hline
    \end{tabular}
    \caption{Sets of $(s,t)$-uniform simplicial complexes with the same limit of the ratio of sphere area over sphere length}
    \label{tab:complexes-same-ratio}
\end{table}
\renewcommand{\arraystretch}{1}

In Figure \ref{fig:sphere-area-length-limit-s-vertex} we represent, for some (s,t)-uniform complexes, the ratio of the area over the length of spheres centered at an $s$-vertex and whose radii vary from $0$ to $20$.
In Figure \ref{fig:sphere-area-length-limit-t-vertex} we represent the same but for spheres centered at a $t$-vertex.
On the horizontal axis we represent the radii of spheres.
On the vertical axis we represent the values of the ratio.
We consider the $(s,t)$-uniform complexes given in Table \ref{tab:complexes-same-ratio}.

These graphs outline that for $s \neq t$ the sequences representing the ratio of sphere area over sphere length are no longer strictly increasing as the radii of spheres increase.
This was the case only for $t$-uniform simplicial complexes (see \cite{L-t-unif}).
Instead each sequence oscillates up and down as it approaches its limit and it may take values above and below the limit.

\begin{figure}[ht]
    \centering
    \includegraphics[width=\textwidth]{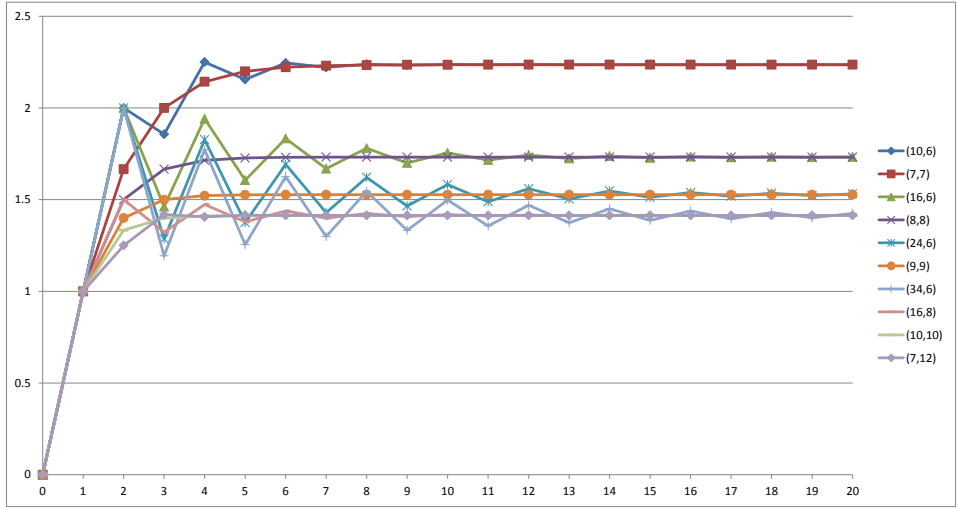}
    \caption{The ratio of sphere area over sphere length for spheres centered at an $s$-vertex}
    \label{fig:sphere-area-length-limit-s-vertex}
\end{figure}

\begin{figure}[ht]
    \centering
    \includegraphics[width=\textwidth]{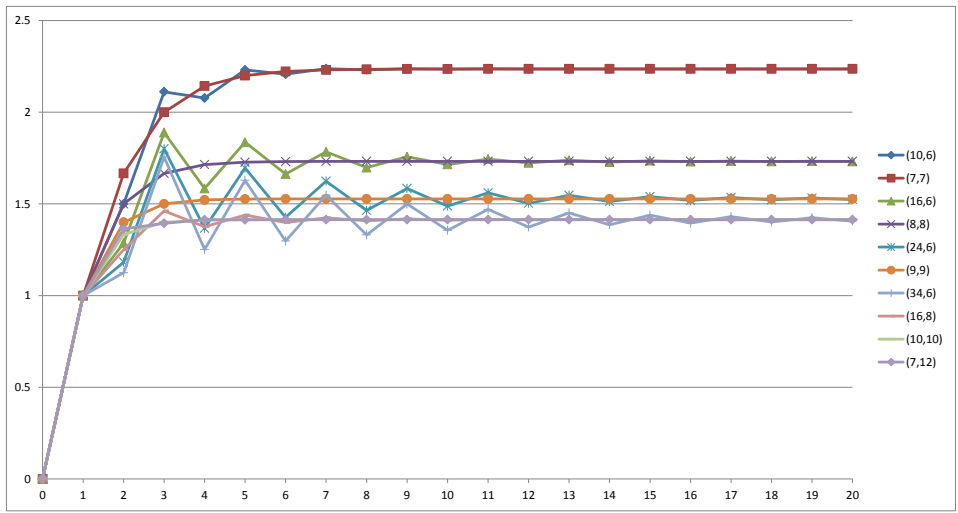}
    \caption{The ratio of sphere area over sphere length for spheres centered at a $t$-vertex}
    \label{fig:sphere-area-length-limit-t-vertex}
\end{figure}

\subsection{Curvature inside spheres}\label{sec:curvature-inside-spheres}

In this section let $s,t$ such that either $s, t \ge 6$, $t$ even except for $s = t = 6$.
We study the Gaussian curvature inside spheres belonging to $(s,t)$-uniform complexes as the radii of spheres grow.
More precisely, we compute the average Gaussian curvature for vertices inside such spheres.

Because $s, t > 6$, the vertices have angle excess.
Thus there is Gaussian curvature at each vertex.
Namely,
\begin{itemize}
    \item $K(v) = 6 - s$, if $v$ is an $s$-vertex;
    \item $K(v) = 6 - t$, if $v$ is a $t$-vertex.
\end{itemize}

For the rest of the paper we consider only the case when spheres are centered at $s$-vertices.
For spheres around $t$-vertices we would obtain similar results.

\begin{theorem}\label{theorem:average-curvature-s-t-unif}
In $(s,t)$-uniform simplicial complexes, $s, t \ge 6$, $t$ even, except for $s = t = 6$, the following holds:
\begin{equation}\label{eq:average-curvature-s-t-uniform}
    \lim_{n\to\infty} K_{Avg}(S_{n}^{s,t}) = 2 - \cfrac{T + \sqrt{T^2 + 8ST}}{4}
\end{equation}
where $T = t-4$ and $S = s-4$.
The limit can be also expressed as
\begin{equation}\label{eq:average-curvature-s-t-uniform-p}
    \lim_{n\to\infty} K_{Avg}(S_{n}^{s,t}) = 2 - P
\end{equation}
where $P = \cfrac{T + \sqrt{T^2 + 8ST}}{4}$.
\end{theorem}

\begin{proof}

We start by finding the general term of the sequence representing the length of spheres.
We denote this sequence by $(|S_{n}^{s,t}|)_{n \geq 0}$.
We consider the sequence $(x_n)_{n \ge -1}$ such that $|S_n^{s,t}| = s  \cdot x_n$.
The initial terms of this sequence are given in Table \ref{tab:initial-terms-s}.
They depend on $s$, $t$, $S$ and $T$.
We use the system of equations given in (\ref{eq:system-hyperbolas}) to express them in terms of $P$ and $R$.
It turns out it is easier to work with these values.

We find the formula of the general term of the sequence $(x_n)_{n \ge -1}$.
Below are the initial terms of the sequence $(|S_n^{s,t}|)_{n \geq 0}$ divided by $s$:
\begin{center}
    \begin{tabular}{|c|c|c|c|c|}
        \hline
        $x_{-1}$ & $x_0$ & $x_1$ & $x_2$ & $x_3$ \\
        \hline
        $-1$ & $0$ & $1$ & $T$ & $2\bigg(\cfrac{T}{2}\bigg)^2 + \bigg(\cfrac{ST}{2}-2\bigg) + 1$ \\
        \hline
        $-1$ & $0$ & $1$ & $2(P-R)$ & $2(P-R)^2 + (PR-2) + 1$ \\
        \hline
    \end{tabular}
\end{center}

In the cases (\ref{case-a}) and (\ref{case-c}), we use Table \ref{tab:coef-x-a-c-v1-v2}.
It contains the coefficients expressed in terms of $v_1$ and $v_2$.
Because $v_1 = x_1 = 1$ and $v_2 = x_2 = 2(P-R)$, the coefficients become
\begin{table}[ht]
    \centering
    \begin{tabular}{|c|c|c|c|}
        \hline
        $A_S$ & $B_S$ & $C_S$ & $D_S$ \\
        \hline
        $\cfrac{2P - R}{(P+R)\sqrt{P^2 - 4}}$
        & $-\cfrac{2P - R}{(P+R)\sqrt{P^2 - 4}}$
        & $\cfrac{2R - P}{(P+R)\sqrt{R^2 - 4}}$
        & $-\cfrac{2R - P}{(P+R)\sqrt{R^2 - 4}}$ \\
        \hline
    \end{tabular}
    \caption{The coefficients of the general term of the sequence $(|S_{n}^{s,t}|)_{n \geq 0}$ in the cases (\ref{case-a}) and (\ref{case-c})}
    \label{tab:sphere-total-distinct}
\end{table}

In the case (\ref{case-b}), we use Table \ref{tab:coef-x-b-v1-v2}.
It contains the coefficients expressed in terms of $v_1$ and $v_2$.
Because $R=2$, $v_1 = x_1 = 1$ and $v_2 = x_2 = 2(P-R)$, the coefficients become
\begin{table}[ht]
    \centering
    \begin{tabular}{|c|c|c|c|}
        \hline
        $A_S$ & $B_S$ & $C_S$ & $D_S$ \\
        \hline
        $\cfrac{2P - 2}{(P+2)\sqrt{P^2-4}}$
        & $-\cfrac{2P - 2}{(P+2)\sqrt{P^2-4}}$
        & $0$
        & $-\cfrac{4-P}{P+2}$ \\
        \hline
    \end{tabular}
    \caption{The coefficients of the general term of the sequence $(|S_{n}^{s,t}|)_{n \geq 0}$ in the case (\ref{case-b})}
    \label{tab:sphere-total-duplicate}
\end{table}

Next we find the general term of the sequence representing the number of $s$-vertices on spheres.
We denote this sequence by $(|V_n^{s,t}|)_{n \geq 0}$.
We consider a sequence $(x_n)_{n \ge -1}$, such that $|V_n^{s,t}| = s \cdot x_n$.
We find the formula of the general term of the sequence $(x_n)_{n \geq -1}$.
Below we give the initial terms of the sequence representing the number of $s$-vertices on spheres divided by $s$.
\begin{center}
    \begin{tabular}{|c|c|c|c|c|}
        \hline
        $x_{-1}$ & $x_0$ & $x_1$ & $x_2$ & $x_3$ \\
        \hline
        $0$ & $0$ & $0$ & $\cfrac{T}{2}$ & $\bigg(\cfrac{T}{2}\bigg)^2$ \\
        \hline
        $0$ & $0$ & $0$ & $P-R$ & $(P-R)^2$ \\
        \hline
    \end{tabular}
\end{center}

In the cases (\ref{case-a}) and (\ref{case-c}) we use Table \ref{tab:coef-x-a-c-v1-v2}.
It contains the coefficients expressed in terms of $v_1$ and $v_2$.
Because $v_1 = x_1 = 0$ and $v_2 = x_2 = P-R$, the coefficients become
\begin{table}[ht]
    \centering
    \begin{tabular}{|c|c|c|c|}
        \hline
        $A_V$ & $B_V$ & $C_V$ & $D_V$ \\
        \hline
        $\cfrac{P - R}{(P+R)\sqrt{P^2 - 4}}$
        & $-\cfrac{P - R}{(P+R)\sqrt{P^2 - 4}}$
        & $\cfrac{R - P}{(P+R)\sqrt{R^2 - 4}}$
        & $-\cfrac{R - P}{(P+R)\sqrt{R^2 - 4}}$ \\
        \hline
    \end{tabular}
    \caption{The coefficients of the general term of the sequence $(|V_{n}^{s,t}|)_{n \geq 0}$ in the cases (\ref{case-a}) and (\ref{case-c})}
    \label{tab:sphere-s-vertices-distinct}
\end{table}

In the case (\ref{case-b}) we use Table \ref{tab:coef-x-b-v1-v2}.
It contains the coefficients expressed in terms of $v_1$ and $v_2$.
Because $R=2$, $v_1 = x_1 = 0$ and $v_2 = x_2 = P-R$, the coefficients become
\begin{table}[ht]
    \centering
    \begin{tabular}{|c|c|c|c|}
        \hline
        $A_V$ & $B_V$ & $C_V$ & $D_V$ \\
        \hline
        $\cfrac{P - 2}{(P+2)\sqrt{P^2-4}}$
        & $-\cfrac{P - 2}{(P+2)\sqrt{P^2-4}}$
        & $0$
        & $-\cfrac{2-P}{P+2}$ \\
        \hline
    \end{tabular}
    \caption{The coefficients of the general term of the sequence $(|V_{n}^{s,t}|)_{n \geq 0}$ in the case (\ref{case-b})}
    \label{tab:sphere-s-vertices-duplicate}
\end{table}

The tables \ref{tab:sphere-total-distinct} and \ref{tab:sphere-total-duplicate} contain
the general terms of the sequences representing the sphere lengths.
Besides in the tables \ref{tab:sphere-s-vertices-distinct} and \ref{tab:sphere-s-vertices-duplicate} the general terms of the sequences representing the number of $s$-vertices on spheres are given.
Then, since $|S_n^{s,t}| = |V_n^{s,t}| + |W_n^{s,t}|$, we can find the general term of the sequence representing the number of $t$-vertices on spheres.
We denote this sequence by $(|W_n^{s,t}|)_{n \geq 0}$.
\begin{table}[ht]
    \centering
    \begin{tabular}{|c|c|c|c|}
        \hline
        $A_W$ & $B_W$ & $C_W$ & $D_W$ \\
        \hline
        $\cfrac{P}{(P+R)\sqrt{P^2 - 4}}$
        & $-\cfrac{P}{(P+R)\sqrt{P^2 - 4}}$
        & $\cfrac{R}{(P+R)\sqrt{R^2 - 4}}$
        & $-\cfrac{R}{(P+R)\sqrt{R^2 - 4}}$ \\
        \hline
    \end{tabular}
    \caption{The coefficients of the general term of the sequence $(|W_{n}^{s,t}|)_{n \geq 0}$ in the cases (\ref{case-a}) and (\ref{case-c})}
    \label{tab:sphere-t-vertices-distinct}
\end{table}

\begin{table}[ht]
    \centering
    \begin{tabular}{|c|c|c|c|}
        \hline
        $A_W$ & $B_W$ & $C_W$ & $D_W$ \\
        \hline
        $\cfrac{P}{(P+2)\sqrt{P^2-4}}$
        & $-\cfrac{P}{(P+2)\sqrt{P^2-4}}$
        & $0$
        & $-\cfrac{2}{P+2}$ \\
        \hline
    \end{tabular}
    \caption{The coefficients of the general term of the sequence $(|W_{n}^{s,t}|)_{n \geq 0}$ in the case (\ref{case-b})}
    \label{tab:sphere-t-vertices-duplicate}
\end{table}

In the case (\ref{case-b}) we have $R = 2$.
So in all cases ((\ref{case-a}), (\ref{case-b}) and (\ref{case-c})) we get
$$A_S = \cfrac{2P - R}{(P+R)\sqrt{P^2 - 4}},
A_V = \cfrac{P - R}{(P+R)\sqrt{P^2 - 4}},
A_W = \cfrac{P}{(P+R)\sqrt{P^2 - 4}}.$$

The Gaussian curvature inside a sphere $S_{n}^{s,t}$ is equal to the sum of the Gaussian curvatures at the vertices inside this sphere.
So for vertices inside $S_{n}^{s,t}$ we have:
\begin{itemize}
    \item the number of vertices $= 1 + \sum\limits_{k=0}^{n-1} |S_k^{s,t}|$,
    \item the number of $s$-vertices $= 1 + \sum\limits_{k=0}^{n-1} |V_k^{s,t}|$,
    \item the number of $t$-vertices $= \sum\limits_{k=0}^{n-1} |W_k^{s,t}|$
\end{itemize}
Because the sphere $S_0^{s,t}$ has length $0$, we have to count the center $s$-vertex as well.
The first two expressions above have therefore an extra vertex added to the sum.
The constant $1$ plays no role, however, when we compute the limit below.

The curvature inside $S_{n}^{s,t}$ is
$$K(S_{n}^{s,t}) = (6-s) \bigg( 1 + \sum_{k=0}^{n-1} |V_k^{s,t}| \bigg) + (6-t)\sum_{k=0}^{n-1} |W_k^{s,t}|$$
For each vertex inside $S_{n}^{s,t}$, the average curvature is
$$K_{Avg}(S_{n}^{s,t}) = \cfrac{(6-s) \bigg( 1 + \sum\limits_{k=0}^{n-1} |V_k^{s,t}| \bigg) +
(6-t)\sum\limits_{k=0}^{n-1} |W_k^{s,t}|}
{1 + \sum\limits_{k=0}^{n-1} |S_k^{s,t}|}$$

The sequences $(x_{n})_{n \geq 0} \in \{(|S_{n}^{s,t}|)_{n \geq 0}, (|V_{n}^{s,t}|)_{n \geq 0}, (|W_{n}^{s,t}|)_{n \geq 0}\}$ follow the same recurrence relation.
By (\ref{eq:limit-sum-over-last-p}), we have 
$$\lim\limits_{n\to\infty} \cfrac{\sum\limits_{k = 0}^n x_k}{x_n} =
\cfrac{1 + \sqrt{\cfrac{P+2}{P-2}}}{2}$$
So 
$$\lim\limits_{n\to\infty} K_{Avg}(S_{n}^{s,t}) =
\lim\limits_{n\to\infty} \cfrac{(6-s) \cdot \bigg(1 + \sum\limits_{k=0}^{n-1} |V_k^{s,t}|\bigg) +
(6-t) \cdot \sum\limits_{k=0}^{n-1} |W_k^{s,t}|}{1 + \sum\limits_{k=0}^{n-1} |S_k^{s,t}|} =$$
$$= \lim\limits_{n\to\infty} \cfrac{(6-s) \cdot |V_{n-1}^{s,t}| \cdot \cfrac{\sum\limits_{k=0}^{n-1} |V_k^{s,t}|}{|V_{n-1}^{s,t}|}
+ (6-t)\cdot |W_{n-1}^{s,t}| \cdot \cfrac{\sum\limits_{k=0}^{n-1} |W_k^{s,t}|}{|W_{n-1}^{s,t}|}}
{|S_{n-1}^{s,t}| \cdot \cfrac{\sum\limits_{k=0}^{n-1} |S_k^{s,t}|}{|S_{n-1}^{s,t}|}} =$$
$$= \cfrac{(6-s) \cdot \lim\limits_{n\to\infty} |V_{n-1}^{s,t}| \cdot \lim\limits_{n\to\infty} \cfrac{\sum\limits_{k=0}^{n-1} |V_k^{s,t}|}{|V_{n-1}^{s,t}|}
+ (6-t) \cdot \lim\limits_{n\to\infty} |W_{n-1}^{s,t}| \cdot \lim\limits_{n\to\infty} \cfrac{\sum\limits_{k=0}^{n-1} |W_k^{s,t}|}{|W_{n-1}^{s,t}|}}
{\lim\limits_{n\to\infty} |S_{n-1}^{s,t}| \cdot \lim\limits_{n\to\infty} \cfrac{\sum\limits_{k=0}^{n-1} |S_k^{s,t}|}{|S_{n-1}^{s,t}|}} =$$
$$= \cfrac{(6-s) \cdot \lim\limits_{n\to\infty} |V_{n-1}^{s,t}| \cdot \cfrac{1 + \sqrt{\cfrac{P+2}{P-2}}}{2}
+ (6-t) \cdot \lim\limits_{n\to\infty} |W_{n-1}^{s,t}| \cdot \cfrac{1 + \sqrt{\cfrac{P+2}{P-2}}}{2}}
{\lim\limits_{n\to\infty} |S_{n-1}^{s,t}| \cdot \cfrac{1 + \sqrt{\cfrac{P+2}{P-2}}}{2}} =$$
$$ = \lim\limits_{n\to\infty} \cfrac{(6-s) \cdot |V_{n-1}^{s,t}| + (6-t) \cdot |W_{n-1}^{s,t}|}{|S_{n-1}^{s,t}|} = $$
$$= \lim\limits_{n\to\infty} \Bigg[ \cfrac{(6-s) \cdot (A_V a^n + B_V b^n + C_V c^n + D_V d^n)}{A_S a^n + B_S b^n + C_S c^n + D_S d^n} +$$
$$+ \cfrac{(6-t) \cdot (A_W a^n + B_W b^n + C_W c^n + D_W d^n)}
{A_S a^n + B_S b^n + C_S c^n + D_S d^n} \Bigg] =$$
$$= \lim\limits_{n\to\infty} \Vast\{ \cfrac{
(6-s) \cdot \bigg[ A_V + B_V \bigg(\cfrac{b}{a}\bigg)^n  + C_V \bigg(\cfrac{c}{a}\bigg)^n  + D_V \bigg(\cfrac{d}{a}\bigg)^n \bigg]}
{A_S + B_S \bigg(\cfrac{b}{a}\bigg)^n + C_S \bigg(\cfrac{c}{a}\bigg)^n + D_S \bigg(\cfrac{d}{a}\bigg)^n} +$$
$$+ \cfrac{(6-t)\cdot \bigg[ A_W + B_W \bigg(\cfrac{b}{a}\bigg)^n + C_W \bigg(\cfrac{c}{a}\bigg)^n + D_W \bigg(\cfrac{d}{a}\bigg)^n \bigg]}
{A_S + B_S \bigg(\cfrac{b}{a}\bigg)^n + C_S \bigg(\cfrac{c}{a}\bigg)^n + D_S \bigg(\cfrac{d}{a}\bigg)^n} \Vast\}=$$
$$= \cfrac{(6-s) A_V + (6-t) A_W}{A_S} =$$
$$= \cfrac{(6-s) (P-R) + (6-t) P}{2P-R}$$

As $S = s-4$ and $T = t-4$, we get $6-s = 2-S$ and $6-t = 2-T$.
Then, due to (\ref{eq:system-hyperbolas}), we have \begin{center}
$\begin{cases}
    P - R = \cfrac{T}{2} \\
    PR = \cfrac{ST}{2}
\end{cases}$ 
$\implies$
$\begin{cases}
    T = 2(P - R) \\
    S = \cfrac{PR}{P-R}
\end{cases}$.\end{center}
So 
$$\lim\limits_{n\to\infty} K_{Avg}(S_{n}^{s,t}) = \cfrac{(2-S) (P-R) + (2-T) P}{2P-R} =$$
$$= \cfrac{\bigg(2-\cfrac{PR}{P-R}\bigg) (P-R) + [2-2(P - R)] P}{2P-R} =$$
$$= \cfrac{2(P-R) - PR - 2P(P - R - 1)}{2P-R} =$$
$$= \cfrac{2P - 2R - PR - 2P^2 + 2PR + 2P}{2P-R} =$$
$$= \cfrac{4P - 2R - 2P^2 + PR}{2P-R} =$$
$$= \cfrac{2(2P - R) - P(2P - R)}{2P-R} =$$
$$= 2 - P$$

Replacing $P$ with the value given in (\ref{eq:p-in-terms-of-s-t}) completes the proof.
\end{proof}

The $(s,t)$-uniform complexes on the same red and green hyperbolas in Figure \ref{fig:table-hyperbolas-s-t} correspond to the same value of $P$.
This implies that the limit as $n\to\infty$ of the average curvature inside spheres $S_{n}^{s,t}$ is constant.

For $t$ odd, note that in a $t$-uniform complex a $t$-vertex has constant curvature $6-t$.
So the average curvature inside any loop in such complex equals $6-t$.
As before, one can identify on the graph a few sets of $(s,t)$-uniform complexes with the same limit of the average curvature inside spheres.
Namely,

\renewcommand{\arraystretch}{2}
\begin{table}[ht]
    \begin{tabular}{|l|c|}
        \hline
        $(s,t)$-uniform complexes & $\lim\limits_{n\to\infty} K_{Avg}(S_{n}^{s,t})$ \\
        \hline
        $(10, 6), (7, 7)$ & $-1$ \\
        \hline
        $(16, 6), (8, 8)$ & $-2$ \\
        \hline
        $(24, 6), (9, 9)$ & $-3$ \\
        \hline
        $(34, 6), (16, 8)$, $(10, 10), (7, 12)$ & $-4$ \\
        \hline
    \end{tabular}
    \caption{Sets of $(s,t)$-uniform simplicial complexes with the same limit of the average curvature inside spheres}
    \label{tab:complexes-same-average-curvature}
\end{table}
\renewcommand{\arraystretch}{1}

\end{document}